\RequirePackage{etoolbox}
\csdef{input@path}{{style/}{graphics/}}
\documentclass[aap,MSNbibl,seceqn,dvips]{arximspdf}
\usepackage{url,breakurl}
\doi{10.1214/14-AAP1003} %kopijuoti is PTS
\volume{25}
\issue{2}
\pubyear{2015}
\firstpage{548}
\lastpage{581}

\makeatletter
\newcommand{\rrvert}{\vert}
\newcommand{\llvert}{\vert}
\newcommand{\eqref}[1]{(\ref{#1})}
\newtheorem{theorem}{Theorem}
\newtheorem{lemma}[theorem]{Lemma}
\newproclaim{remark}[theorem]{Remark}
\makeatother

\begin{document}
\begin{frontmatter}

\title{Large deviations for Markovian nonlinear Hawkes~processes}
\runtitle{LDP for Markovian nonlinear Hawkes processes}

\begin{aug}
\author[A]{\fnms{Lingjiong} \snm{Zhu}\corref{}\ead[label=e1]{ling@cims.nyu.edu}\thanksref{T1}}
\thankstext{T1}{Supported by NSF Grant DMS-09-04701, DARPA grant and
MacCracken Fellowship at New York University.}
\runauthor{L. Zhu}
\affiliation{New York University}
\address[A]{Courant Institute of Mathematical Sciences\\
New York University\\
251 Mercer Street\\
New York, New York 10012\\
USA\\
\printead{e1}} %adresu isvedimo komanda gale!
\end{aug}

% HISTORY:
\received{\smonth{6} \syear{2012}}
\revised{\smonth{12} \syear{2013}}

% ABSTRACT
%
\begin{abstract}
Hawkes process is a class of simple point processes that is
self-exciting and has clustering effect.
The intensity of this point process depends on its entire past history.
It has wide applications
in finance, neuroscience and many other fields.
In this paper, we study the large deviations for nonlinear Hawkes processes.
The large deviations for linear Hawkes processes has been studied by
Bordenave and Torrisi.
In this paper, we prove first a large deviation principle for a special
class of nonlinear Hawkes processes,
that is, a Markovian Hawkes process with nonlinear rate and exponential
exciting function,
and then generalize it to get the result for sum of exponentials
exciting functions.
We then provide an alternative proof for the large deviation principle
for a linear Hawkes process.
Finally, we use an approximation approach to prove the large deviation
principle for a special class of nonlinear Hawkes processes
with general exciting functions.
\end{abstract}

% KEYWORDS
\begin{keyword}[class=AMS]
\kwd{60G55}
\kwd{60F10}
\end{keyword}
\begin{keyword}
\kwd{Large deviations}
\kwd{rare events}
\kwd{point processes}
\kwd{Hawkes processes}
\kwd{self-exciting processes}
\end{keyword}

\end{frontmatter}

%s1 #&#
\section{Introduction}\label{sec1}

Let $N$ be a simple point process on $\mathbb{R}$, and let $\mathcal
{F}_{t}:=\sigma(N(C),C\in\mathcal{B}(\mathbb{R}), C\subset(-\infty,t])$ be
an increasing family of $\sigma$-algebras. Any nonnegative $\mathcal
{F}_{t}$-progressively measurable process $\lambda_{t}$ with\vspace*{-1pt}
%
%e1.1 #&#
\begin{equation}
\mathbb{E} \bigl[N(a,b]|\mathcal{F}_{a} \bigr]=\mathbb{E} \biggl[
\int_{a}^{b}\lambda_{s}\,ds \Big|
\mathcal{F}_{a} \biggr]\vspace*{-1pt}
\end{equation}
a.s. for all intervals $(a,b]$ is called an $\mathcal{F}_{t}$-intensity
of $N$. We use the notation $N_{t}:=N(0,t]$ to denote the number of
points in the interval $(0,t]$.

A general Hawkes process is a simple point process $N$ admitting an
$\mathcal{F}_{t}$-intensity\vspace*{-1pt}
%
%e1.2 #&#
\begin{equation}
\lambda_{t}:=\lambda\biggl(\int_{0}^{t}h(t-s)N(ds)
\biggr),\vspace*{-1pt}
\end{equation}
where $\lambda(\cdot)\dvtx \mathbb{R}^{+}\rightarrow\mathbb{R}^{+}$ is
locally integrable
and left continuous, $h(\cdot)\dvtx \mathbb{R}^{+}\rightarrow\mathbb
{R}^{+}$, and
we always assume that $\Vert h\Vert_{L^{1}}=\int_{0}^{\infty
}h(t)\,dt<\infty$.
The notation $\int_{0}^{t}h(t-s)N(ds)$ stands for $\int_{(0,t)}h(t-s)N(ds)$.
Local integrability assumption of $\lambda(\cdot)$ ensures that the
process is nonexplosive and left continuity assumption
ensures that $\lambda_{t}$ is $\mathcal{F}_{t}$-predictable.

In the literature, $h(\cdot)$ and $\lambda(\cdot)$ are usually referred to
as exciting function and rate function, respectively.

Let $Z_{t}=\sum_{0<\tau_{j}<t}h(t-\tau_{j})$,
where $\tau_{j}$ is the $j$th arrival time of the process for $j\geq
1$. Thus we can write $\lambda_{t}=\lambda(Z_{t})$.

This is known as the nonlinear Hawkes process; see Br\'{e}maud and
Massouli\'{e}~\cite{Bremaud}. When the exciting function $h(\cdot)$
is exponential or a sum of exponentials, the process is Markovian, and
we name it a Markovian nonlinear Hawkes process.

When $\lambda(\cdot)$ is linear, this is known as the (linear) Hawkes
process, which was introduced in Hawkes \cite{Hawkes}.
If $\lambda(\cdot)$ is linear and $h(\cdot)$ is exponential or a sum of
exponentials,
the (linear) Markovian Hawkes process is sometimes referred to as
Markovian self-exciting processes; see, for example, Oakes \cite{Oakes}.
You can think of the arrival times $\tau_{j}$ as ``bad'' events,
which can be the arrivals of claims in insurance literature or the time
of defaults of big firms in the real world.
Hawkes process captures both the self-exciting property and the
clustering effect,
which explains why it has wide applications in cosmology, ecology,
epidemiology, seismology, neuroscience and DNA modeling.
For a list of references to these applications, see Bordenave and
Torrisi \cite{Bordenave}.

Hawkes process has also been applied in finance.
Empirical comparisons suggest that Hawkes processes have some of the
typical characteristics of a \mbox{financial} time series.
Financial data have been analyzed using Hawkes processes. Self-exciting
processes are used for the calculation of conditional risk measures,
such as the value-at-risk. Another area of finance where Hawkes
processes have been considered is credit default modeling.
Hawkes processes have been proposed as models for the arrival of
company defaults in a bond portfolio. For a list of references to the
applications in finance,
see Liniger \cite{Liniger} and Zhu \cite{ZhuThesis}.

For a short history of Hawkes process, we refer to Liniger \cite{Liniger}.
For a survey on Hawkes processes and related self-exciting processes,
Poisson cluster processes, marked point processes, etc.,
we refer to Daley and Vere-Jones \cite{Daley}.

When $\lambda(\cdot)$ is linear, say $\lambda(z)=\nu+z$, then one can
use immigration-birth representation,
also known as Galton--Watson theory to study it. Under the
immigration-birth representation,
if the immigrants are distributed as Poisson process with \mbox{intensity} $\nu
$ and each immigrant generates a cluster whose number of points is
denoted by $S$,
then $N_{t}$ is the total number of points generated in the clusters up
to time $t$. If the process is ergodic, we have
%
%e1.3 #&#
\begin{equation}
\lim_{t\rightarrow\infty}\frac{N_{t}}{t}=\nu\mathbb{E}[S]\qquad\mbox{a.s.}
\end{equation}

The central limit theorem for linear Hawkes processes was obtained in
Bacry et al. \cite{Bacry},
and it was proved for nonlinear Hawkes processes in Zhu \cite{ZhuCLT}.
The moderate deviations for linear
Hawkes processes was obtained in Zhu \cite{ZhuMDP}.

Bordenave and Torrisi \cite{Bordenave} proves that if $0<\mu=\int
_{0}^{\infty}h(t)\,dt<1$ and\break $\int_{0}^{\infty}th(t)\,dt<\infty$,
then $ (\frac{N_{t}}{t}\in\cdot)$ satisfies\vspace*{1pt} the large deviation
principle (LDP) with the good rate function $I(\cdot)$, that is,
for any closed set $C\subset\mathbb{R}$,
%
%e1.4 #&#
\begin{equation}
\limsup_{t\rightarrow\infty}\frac{1}{t}\log\mathbb{P}(N_{t}/t
\in C)\leq-\inf_{x\in C}I(x),
\end{equation}
and for any open set $G\subset\mathbb{R}$,
%
%e1.5 #&#
\begin{equation}
\liminf_{t\rightarrow\infty}\frac{1}{t}\log\mathbb{P}(N_{t}/t
\in G)\geq-\inf_{x\in G}I(x),
\end{equation}
where
%
%e1.6 #&#
\begin{equation}
I(x)= \cases{ \displaystyle x\theta_{x}+\nu-\frac{\nu x}{\nu+\mu x},
&\quad
if $x\in[0,\infty)$,
\cr
+\infty, &\quad otherwise,}
\end{equation}
where $\theta=\theta_{x}$ is the unique solution in $(-\infty,\mu-1-\log
\mu]$ of $\mathbb{E}[e^{\theta S}]=\frac{x}{\nu+x\mu}$, $x>0$.
It is well known that (e.g., see page 39 of Jagers \cite
{Jagers}), for all $\theta\in(-\infty,\mu-1-\log\mu]$, $\mathbb
{E}[e^{\theta S}]$ satisfies
%
%e1.7 #&#
\begin{equation}
\mathbb{E}\bigl[e^{\theta S}\bigr]=e^{\theta}\exp\bigl\{\mu\bigl(
\mathbb{E}\bigl[e^{\theta S}\bigr]-1\bigr) \bigr\}.
\end{equation}
See Dembo and Zeitouni \cite{Dembo} for general background regarding
large deviations and the applications.
Also Varadhan \cite{Varadhan} has an excellent survey article on this subject.

In a recent paper, Zhu \cite{ZhuCIR} studied the limit theorems for a
Cox--Ingersoll--Ross process with Hawkes jumps,
an extension of the linear Hawkes processes. Karabash and Zhu \cite
{Karabash} obtained to the limit theorems for
linear marked Hawkes processes, another extension of the classical
Hawkes processes.

The large deviations result for $(N_{t}/t\in\cdot)$ is helpful to study
the ruin probabilities
of a risk process when the claims arrivals follow a Hawkes process.
Stabile and Torrisi \cite{Stabile} considered risk processes with
nonstationary Hawkes claims arrivals and studied the asymptotic
behavior of infinite and finite horizon ruin probabilities under
light-tailed conditions on the claims.
The corresponding result for heavy-tailed claims was obtained by Zhu
\cite{ZhuRuin}.

In this paper, we are interested in Hawkes processes with general
nonlinear $\lambda(\cdot)$.
If $\lambda(\cdot)$ is nonlinear,
then the usual Galton--Watson theory approach no longer works. If the
exciting function $h$ is exponential or a sum of exponentials,
the process is Markovian, and there exists a generator of the process.
The difficulty arises when $h$ is not exponential or a sum of exponentials
in which case the process is non-Markovian. Another possible
generalization is to consider $h$ to be random.
Then, we will get a marked point process. For a discussion on marked
point processes, see Cox and Isham \cite{Cox}.

When $\lambda(\cdot)$ is nonlinear, Br\'{e}maud and Massouli\'{e} \cite
{Bremaud} proves that under certain conditions,
there exists a unique stationary version of the nonlinear Hawkes
process and Br\'{e}maud and Massouli\'{e} \cite{Bremaud}
also proves the convergence to equilibrium of a nonstationary version,
both in distribution and in variation.

In this paper, we will prove the large deviation when $h$ is
exponential, and $\lambda$ is nonlinear first.
Then, we will generalize the proof to the case when $h$ is a sum of
exponentials. We will use that to recover the result proved in
Bordenave and Torrisi \cite{Bordenave}. Finally, we will prove the
result for a special class of nonlinear $\lambda$ and general $h$.

%s2 #&#
\section{An ergodic lemma}\label{sec2}\label{ErgodicSection}

In this section, we prove an ergodic theorem for a class of Markovian
processes with jumps more general than the Markovian nonlinear Hawkes processes.

Let $Z_{i}(t):=\sum_{\tau_{j}<t}a_{i}e^{-b_{i}(t-\tau_{j})}$, $1\leq
i\leq d$, where
$b_{i}>0$, $a_{i}\neq0$ (might be negative), and $\tau_{j}$'s are the arrivals
of the simple point process with intensity $\lambda(Z_{1}(t),\ldots,Z_{d}(t))$ at time $t$, where
$\lambda\dvtx \mathcal{Z}\rightarrow\mathbb{R}^{+}$ and $\mathcal{Z}:=\mathbb
{R}^{\varepsilon_{1}}\times\cdots\times \mathbb{R}^{\varepsilon_{d}}$
is the domain for $(Z_{1}(t),\ldots,Z_{d}(t))$,
where $\mathbb{R}^{\varepsilon_{i}}:=\mathbb{R}^{+}$ or $\mathbb{R}^{-}$
depending on whether $\varepsilon_{i}=+1$ or $-1$,
where $\varepsilon_{i}=+1$ if $a_{i}>0$ and $\varepsilon_{i}=-1$ otherwise.
If we assume the exciting function to be $h(t)=\sum_{i=1}^{d}a_{i}e^{-b_{i}t}$,
then a Markovian nonlinear Hawkes process is a simple point process
with intensity of the form $\lambda(\sum_{i=1}^{d}Z_{i}(t))$.

The generator $\mathcal{A}$ for $(Z_{1}(t),\ldots,Z_{d}(t))$ is given by
%
%e2.1 #&#
\begin{eqnarray}
\mathcal{A}f &=&-\sum_{i=1}^{d}b_{i}z_{i}
\frac{\partial f}{\partial z_{i}}
\nonumber\\[-8pt]\\[-8pt]
&&{} +\lambda(z_{1},\ldots,z_{d})
\bigl[f(z_{1}+a_{1},\ldots,z_{d}+a_{d})-f(z_{1},
\ldots,z_{d})\bigr].\nonumber
\end{eqnarray}
For a reference to generators for Markov processes with jumps, see
Davis \cite{Davis}.

We want to prove the existence and uniqueness of the invariant
probability measure for $(Z_{1}(t),\ldots,Z_{d}(t))$.
Here the invariance is in time.

%le1 #&#
\begin{lemma}\label{ergodiclemma}
Consider $h(t)=\sum_{i=1}^{d}a_{i}e^{-b_{i}t}>0$.
Assume $\lambda(z_{1},\ldots,z_{n})\leq\sum_{i=1}^{d}\alpha
_{i}|z_{i}|+\beta$, where\vspace*{1pt} $\beta>0$ and $\alpha_{i}>0$, $1\leq i\leq d$,
satisfies $\sum_{i=1}^{d}\frac{|a_{i}|}{b_{i}}\alpha_{i}<1$. Then,
there exists a unique invariant probability measure for
$(Z_{1}(t),\ldots,\break Z_{d}(t))$.
\end{lemma}

\begin{pf}
The lecture notes \cite{Hairer} by Martin Hairer gives the criterion
for the existence of an invariant probability measure for Markov processes.
Suppose we have a jump diffusion process with generator $\mathcal{L}$.
If we can find $u$ such that \mbox{$u\geq0$}, $\mathcal{L}u\leq C_{1}-C_{2}u$
for some constants $C_{1},C_{2}>0$,
then there exists an invariant probability measure.

Try $u(z_{1},\ldots,z_{d})=\sum_{i=1}^{d}\varepsilon_{i}c_{i}z_{i}\geq0$,
where $c_{i}>0$, $1\leq i\leq d$. Then
%
%e2.2 #&#
\begin{eqnarray}
\mathcal{A}u&=&-\sum_{i=1}^{d}b_{i}
\varepsilon_{i}c_{i}z_{i}+\lambda(z_{1},
\ldots,z_{d})\sum_{i=1}^{d}a_{i}
\varepsilon_{i}c_{i}
\nonumber\\[-8pt]\\[-8pt]
&\leq&-\sum_{i=1}^{d}b_{i}c_{i}|z_{i}|+
\sum_{i=1}^{d}\alpha_{i}|z_{i}|
\sum_{i=1}^{d}|a_{i}|c_{i}+
\beta\sum_{i=1}^{d}|a_{i}|c_{i}.
\nonumber
\end{eqnarray}
Taking $c_{i}=\frac{\alpha_{i}}{b_{i}}>0$, we get
%
%e2.3 #&#
\begin{eqnarray}
\mathcal{A}u&\leq&- \Biggl(1-\sum_{i=1}^{d}
\frac{|a_{i}|\alpha_{i}}{b_{i}} \Biggr) \sum_{i=1}^{d}
\alpha_{i}|z_{i}|+\beta\sum_{i=1}^{d}
\frac{|a_{i}|\alpha_{i}}{b_{i}}
\nonumber\\[-8pt]\\[-8pt]
&\leq&-\min_{1\leq i\leq d}b_{i}\cdot\Biggl(1-\sum
_{i=1}^{d}\frac{|a_{i}|\alpha_{i}}{b_{i}} \Biggr)u+\beta\sum
_{i=1}^{d}\frac{|a_{i}|\alpha_{i}}{b_{i}}.
\nonumber
\end{eqnarray}

Next, we will prove the uniqueness of the invariant probability measure.
It is sufficient to prove that for any $x,y\in\mathcal{Z}_{d}$, there
exist times $T_{1},T_{2}>0$ such that
$\mathcal{P}^{x}(T_{1},\cdot)$ and $\mathcal{P}^{y}(T_{2},\cdot)$ are
not mutually singular.
Here $\mathcal{P}^{x}(T,\cdot):=\mathbb{P}(Z^{x}_{T}\in\cdot)$,
where $Z^{x}_{T}$ is $Z_{T}$ starting at $Z_{0}=x$, that is,
$Z^{x}_{T}=xe^{-bT}+\break\sum_{\tau_{j}<T}ae^{-b(T-\tau_{j})}$.
To see this, let us prove by contradiction. If there were two distinct
invariant probability measures $\mu_{1}$
and $\mu_{2}$, then there exist two disjoints sets $E_{1}$ and $E_{2}$
such that $\mu_{1}\dvtx E_{1}\rightarrow E_{1}$
and $\mu_{2}\dvtx E_{2}\rightarrow E_{2}$; see, for example, Varadhan \cite
{VaradhanII}. Now, we can choose $x_{1}\in E_{1}$ and $x_{2}\in E_{2}$.
So that $\mathcal{P}^{x_{1}}(T_{1},\cdot)$
and $\mathcal{P}^{x_{2}}(T_{2},\cdot)$ are supported on $E_{1}$ and
$E_{2}$, respectively, for any $T_{1},T_{2}>0$, which implies
that $\mathcal{P}^{x_{1}}(T_{1},\cdot)$ and $\mathcal
{P}^{x_{2}}(T_{2},\cdot)$ are mutually singular.
This leads to a contradiction.

Consider the simplest case $h(t)=ae^{-bt}$.
Let us assume that $x>y>0$. Conditioning on the event that $Z_{t}^{x}$
and $Z_{t}^{y}$ have exactly one jump during the time interval $(0,T)$,
respectively,
the laws of $\mathcal{P}^{x}(T,\cdot)$ and $\mathcal{P}^{y}(T,\cdot)$
have positive densities on the sets
%
%e2.4 #&#
\begin{equation}
\bigl((a+x)e^{-bT},xe^{-bT}+a \bigr)\quad\mbox{and}\quad
\bigl((a+y)e^{-bT},ye^{-bT}+a \bigr),
\end{equation}
respectively. Choosing $T>\frac{1}{b}\log(\frac{x-y+a}{a})$, we have
%
%e2.5 #&#
\begin{equation}
\bigl((a+x)e^{-bT},xe^{-bT}+a \bigr)\cap\bigl((a+y)e^{-bT},ye^{-bT}+a
\bigr)\neq\varnothing,
\end{equation}
which implies that $\mathcal{P}^{x}(T,\cdot)$ and $\mathcal
{P}^{y}(T,\cdot)$ are not mutually singular.

Similarly, one can show the uniqueness of the invariant probability
measure for the multidimensional case. Indeed, it is easy to see
that for any $x,y\in\mathcal{Z}_{d}$, $Z_{T_{1}}^{x}$ and
$Z_{T_{2}}^{y}$ hit a common point for some $T_{1}$
and $T_{2}$ after possibly different number of jumps. Here
$Z_{t}^{x}:=(Z_{t}^{x_{1}},\ldots,Z_{t}^{x_{d}})\in\mathcal{Z}_{d}$
and $Z_{t}^{y}:=(Z_{t}^{y_{1}},\ldots,Z_{t}^{y_{d}})\in\mathcal{Z}_{d}$,
where $Z_{t}^{x_{i}}=x_{i}e^{-b_{i}t}+\sum_{\tau
_{j}<t}a_{i}e^{-b_{i}(t-\tau_{j})}$, $1\leq i\leq d$.
Since $\mathcal{P}^{x}(T_{1},\cdot)$ and $\mathcal{P}^{y}(T_{2},\cdot)$
have probability densities,
$\mathcal{P}^{x}(T_{1},\cdot)$ and $\mathcal{P}^{y}(T_{2},\cdot)$ are
not mutually singular for some $T_{1}$ and $T_{2}$.
\end{pf}

%s3 #&#
\section{Large deviations for Markovian nonlinear Hawkes processes with exponential exciting function}\label{sec3}

We assume first that $h(t)=ae^{-bt}$, where $a,b>0$, that is, the
process $Z_{t}$ jumps upward an amount $a$ at each point
and decays exponentially between points with rate $b$. In this case,
$Z_{t}$ is Markovian.

Notice first that $Z_{0}=0$ and
%
%e3.1 #&#
\begin{equation}
dZ_{t}=-bZ_{t}\,dt+a\,dN_{t},
\end{equation}
which implies that $N_{t}=\frac{1}{a}Z_{t}+\frac{b}{a}\int_{0}^{t}Z_{s}\,ds$.

We prove first the existence of the limit of the logarithmic moment
generating function of $N_{t}$.

%th2 #&#
\begin{theorem}\label{mainlimit}
Assume\vspace*{2pt} that $\lim_{z\rightarrow\infty}\frac{\lambda(z)}{z}=0$ and that
$\lambda(\cdot)$
is continuous and bounded below by some positive constant. Then
%
%e3.2 #&#
\begin{equation}
\lim_{t\rightarrow\infty}\frac{1}{t}\log\mathbb{E}
\bigl[e^{\theta N_{t}}\bigr]=\Gamma(\theta),
\end{equation}
where
%
%e3.3 #&#
\begin{equation}
\Gamma(\theta)=\sup_{(\hat{\lambda},\hat{\pi})\in\mathcal{Q}_{e}} \biggl
\{\int\frac{\theta b}{a}z
\hat{\pi}(dz) +\int(\hat{\lambda}-\lambda)\hat{\pi}(dz)-\int\bigl(\log
(\hat{\lambda}/\lambda) \bigr)\hat{\lambda}\hat{\pi}(dz) \biggr\},\hspace*{-31pt}
\end{equation}
where $\mathcal{Q}_{e}$ is defined as
%
%e3.4 #&#
\begin{equation}
\qquad\mathcal{Q}_{e}= \bigl\{(\hat{\lambda},\hat{\pi})\in\mathcal{Q}\dvtx
\widehat{\mathcal{A}}\mbox{ has unique invariant probability measure }\hat
{\pi} \bigr
\},
\end{equation}
where
%
%e3.5 #&#
\begin{equation}
\mathcal{Q}= \biggl\{(\hat{\lambda},\hat{\pi})\dvtx \hat{\pi}\in\mathcal
{M}\bigl(
\mathbb{R}^{+}\bigr),\int z\hat{\pi}(dz)<\infty, \hat{\lambda}\in
L^{1}(\hat{\pi}), \hat{\lambda}>0 \biggr\},
\end{equation}
where\vspace*{1pt} $\mathcal{M}(\mathbb{R}^{+})$ denotes the space of probability
measures on $\mathbb{R}^{+}$
and for any $\hat{\lambda}$ such that $(\hat{\lambda},\hat{\pi})\in
\mathcal{Q}$, we define the generator $\widehat{\mathcal{A}}$ as
%
%e3.6 #&#
\begin{equation}
\widehat{\mathcal{A}}f(z)=-bz\frac{\partial f}{\partial z}+\hat{\lambda
}(z)\bigl[f(z+a)-f(z)
\bigr],
\end{equation}
for any $f\dvtx \mathbb{R}^{+}\rightarrow\mathbb{R}$ that is $C^{1}$, that
is, continuously differentiable.
\end{theorem}

\begin{pf}
By Lemma \ref{thetafinite}, $\mathbb{E}[e^{\theta N_{t}}]<\infty$ for
any $\theta\in\mathbb{R}$, also
%
%e3.7 #&#
\begin{equation}
\mathbb{E}\bigl[e^{\theta N_{t}}\bigr]=\mathbb{E} \bigl[e^{({\theta
}/{a}) (Z_{t}+b\int_{0}^{t}Z_{s}\,ds )}
\bigr].\label{upper1}
\end{equation}
Define the set
%
%e3.8 #&#
\begin{equation}
\mathcal{U}_{\theta}= \bigl\{u\in C^{1}\bigl(
\mathbb{R}^{+},\mathbb{R}^{+}\bigr)\dvtx  u(z)=e^{f(z)},
\mbox{ where }f\in\mathcal{F} \bigr\},
\end{equation}
where
%
%e3.9 #&#
\begin{eqnarray}
\mathcal{F} &=& \biggl\{f\dvtx  f(z)=Kz+g(z)+L,
\nonumber\\[-9pt]\\[-9pt]
&&\hspace*{6pt}K>\frac{\theta}{a}, K,L\in
\mathbb{R}, g\mbox{ is }C_{1}\mbox{ with compact support} \biggr\}.\nonumber
\end{eqnarray}
Now for any $u\in\mathcal{U}_{\theta}$, define
%
%e3.10 #&#
\begin{equation}
M:=\sup_{z\geq0}\frac{\mathcal{A}u(z)+(({\theta b})/{a})zu(z)}{u(z)}.
\end{equation}
By Dynkin's formula if $M<\infty$, for $V(z):=\frac{\theta b}{a}z$, we have
%
%e3.11 #&#
\begin{eqnarray}
&& \mathbb{E} \bigl[u(Z_{t})e^{\int_{0}^{t}V(Z_{s})\,ds} \bigr]\nonumber
\\
&&\qquad =u(Z_{0})
+\int_{0}^{t}\mathbb{E} \bigl[\bigl(
\mathcal{A}u(Z_{s})+V(Z_{s})u(Z_{s})
\bigr)e^{\int_{0}^{s}V(Z_{v})\,dv} \bigr]\,ds
\\
&&\qquad \leq u(Z_{0})+M\int_{0}^{t}
\mathbb{E} \bigl[u(Z_{s})e^{\int_{0}^{s}V(Z_{v})\,dv} \bigr]\,ds,
\nonumber
\end{eqnarray}
which implies by Gronwall's lemma that
%
%e3.12 #&#
\begin{equation}
\mathbb{E} \bigl[u(Z_{t})e^{\int_{0}^{t}V(Z_{s})\,ds} \bigr]\leq
u(Z_{0})e^{Mt}=u(0)e^{Mt}.\label{upper2}
\end{equation}
Observe that by the definition of $\mathcal{U}_{\theta}$, for any $u\in
\mathcal{U}_{\theta}$,
we have $u(z)\geq c_{1}e^{({\theta}/{a})z}$ for some constant
$c_{1}>0$ and therefore by \eqref{upper1} and \eqref{upper2},
%
%e3.13 #&#
\begin{equation}
\mathbb{E} \bigl[e^{\theta N_{t}} \bigr]\leq\frac{1}{c_{1}} \mathbb{E}
\bigl[u(Z_{t})e^{\int_{0}^{t}(({\theta b})/{a})Z_{s}\,ds} \bigr]\leq\frac
{1}{c_{1}}u(0)e^{Mt}.
\end{equation}
Hence
%
%e3.14 #&#
\begin{equation}
\limsup_{t\rightarrow\infty}\frac{1}{t}\log\mathbb{E}
\bigl[e^{\theta N_{t}} \bigr]\leq M=\sup_{z\geq0}
\frac{\mathcal{A}u(z)+((\theta b)/{a})zu(z)}{u(z)},
\end{equation}
which is still true even if $M=\infty$. Since this holds for any $u\in
\mathcal{U}_{\theta}$,
%
%e3.15 #&#
\begin{equation}
\limsup_{t\rightarrow\infty}\frac{1}{t}\log\mathbb{E}
\bigl[e^{\theta N_{t}} \bigr]\leq\inf_{u\in\mathcal{U}_{\theta}}\sup
_{z\geq0}\frac{\mathcal{A}u(z)+(({\theta b})/{a})zu(z)}{u(z)}.
\end{equation}

Define the tilted probability measure $\widehat{\mathbb{P}}$ by
%
%e3.16 #&#
\begin{equation}
\frac{d\widehat{\mathbb{P}}}{d\mathbb{P}} \bigg|_{\mathcal{F}_{t}} =\exp\biggl\{
\int_{0}^{t}
\bigl(\lambda(Z_{s})-\hat{\lambda}(Z_{s})\bigr)\,ds+\int
_{0}^{t}\log\biggl(\frac{\hat{\lambda}(Z_{s})}{\lambda(Z_{s})}
\biggr)\,dN_{s} \biggr\}. \label{tiltedP}
\end{equation}
Notice that $\widehat{\mathbb{P}}$ defined in \eqref{tiltedP} is indeed a
probability measure by Girsanov formula. (For the theory of absolute continuity
for point processes and their Girsanov formulas, we refer to Lipster
and Shiryaev \cite{Lipster}.)

Now by Jensen's inequality
%
%e3.17 #&#
\begin{eqnarray}
&& \liminf_{t\rightarrow\infty}\frac{1}{t}\log\mathbb{E}
\bigl[e^{\theta N_{t}}\bigr]\nonumber
\\
&&\qquad =\liminf_{t\rightarrow\infty}\frac{1}{t}\log\widehat{\mathbb{E}} \biggl[
\exp\biggl\{\theta N_{t}-\log\frac{d\widehat{\mathbb{P}}}{d\mathbb{P}}
\bigg|_{\mathcal{F}_{t}}
\biggr\} \biggr]
\nonumber\\[-8pt]\\[-8pt]
&&\qquad \geq \liminf_{t\rightarrow\infty}\widehat{\mathbb{E}} \biggl[\frac{1}{t}
\theta N_{t}-\frac{1}{t}\log\frac{d\widehat{\mathbb{P}}}{d\mathbb{P}}
\bigg|_{\mathcal{F}_{t}}
\biggr]
\nonumber
\\
&&\qquad =\liminf_{t\rightarrow\infty}\widehat{\mathbb{E}} \biggl[\frac{1}{t}
\theta N_{t}-\frac{1}{t}\int_{0}^{t}
\bigl(\lambda(Z_{s}) -\hat{\lambda}(Z_{s})\bigr)\,ds-\int
_{0}^{t}\log\biggl(\frac{\hat{\lambda}(Z_{s})}{\lambda(Z_{s})}
\biggr)\,dN_{s} \biggr].\nonumber\hspace*{-20pt}
\end{eqnarray}
Since $N_{t}-\int_{0}^{t}\hat{\lambda}(Z_{s})\,ds$ is a martingale under
$\widehat{\mathbb{P}}$, we have
%
%e3.18 #&#
\begin{equation}
\widehat{\mathbb{E}} \biggl[\int_{0}^{t}\log\biggl(
\frac{\hat{\lambda}(Z_{s})}{\lambda(Z_{s})} \biggr) \bigl(dN_{s}-\hat
{\lambda}(Z_{s})\,ds
\bigr) \biggr]=0.
\end{equation}
Therefore, by the ergodic theorem, (for a reference, see Chapter~16.4
of Koralov and Sinai \cite{Koralov}),
for any $(\hat{\lambda},\hat{\pi})\in\mathcal{Q}_{e}$,
%
%e3.19 #&#
\begin{eqnarray}
\qquad &&\liminf_{t\rightarrow\infty}\frac{1}{t}\log\mathbb{E}
\bigl[e^{\theta N_{t}}\bigr]\nonumber
\\
&&\qquad \geq\liminf_{t\rightarrow\infty}\widehat{\mathbb{E}} \biggl[\frac{1}{t}
\theta N_{t}-\frac{1}{t}\int_{0}^{t}
\bigl(\lambda(Z_{s}) -\hat{\lambda}(Z_{s})\bigr)\,ds
\nonumber\\[-8pt]\\[-8pt]
&&\hspace*{100pt}{}-\int
_{0}^{t}\log\biggl(\frac{\hat{\lambda}(Z_{s})}{\lambda(Z_{s})} \biggr)
\hat{\lambda}(Z_{s})\,ds \biggr]\nonumber
\\
&&\qquad =\int\frac{\theta b}{a}z\hat{\pi}(dz)+\int(\hat{\lambda}-\lambda)\hat
{\pi}(dz)-
\int\bigl(\log(\hat{\lambda}) -\log(\lambda) \bigr)\hat{\lambda}\hat
{\pi}(dz).
\nonumber
\end{eqnarray}
Hence
%
%e3.20 #&#
\begin{eqnarray}
\qquad &&\liminf_{t\rightarrow\infty}\frac{1}{t}\log\mathbb{E}
\bigl[e^{\theta N_{t}}\bigr]
\nonumber\\[-8pt]\\[-8pt]
&&\qquad \geq\sup_{(\hat{\lambda},\hat{\pi})\in\mathcal{Q}_{e}} \biggl\{\int
\frac{\theta b}{a}z\hat{\pi}+
\int(\hat{\lambda}-\lambda)\hat{\pi} -\int\bigl(\log(\hat{\lambda
})-\log(\lambda)
\bigr)\hat{\lambda}\hat{\pi} \biggr\}.
\nonumber
\end{eqnarray}

Recall that
%
%e3.21 #&#
\begin{eqnarray}
\mathcal{F} &=& \biggl\{f\dvtx  f(z)=Kz+g(z)+L, K>\frac{\theta}{a},
\nonumber\\[-9pt]\\[-9pt]
&&\hspace*{6pt} K,L\in \mathbb{R},
g\mbox{ is }C_{1}\mbox{ with compact support} \biggr\}.\nonumber
\end{eqnarray}
We claim that
%
%e3.22 #&#
\begin{equation}
\inf_{f\in\mathcal{F}} \biggl\{\int\widehat{\mathcal{A}}f(z)\hat{\pi}(dz)
\biggr\}= \cases{ 0, &\quad if $(\hat{\lambda},\hat{\pi})\in\mathcal{Q}_{e}$,
\vspace*{3pt}\cr
-\infty, &\quad if $(\hat{\lambda},\hat{\pi})\in\mathcal{Q}\setminus
\mathcal{Q}_{e}$.}
\end{equation}
It is easy to see that for $(\hat{\lambda},\hat{\pi})\in\mathcal
{Q}_{e}$, and $g$ being $C_{1}$ with compact support, $\int\mathcal
{A}g\hat{\pi}=0$.
Next, we can find a sequence $f_{n}(z)\rightarrow z$ pointwise under
the bound $|f_{n}(z)|\leq\alpha z+\beta$, for some $\alpha,\beta>0$, where
$f_{n}(z)$ is $C_{1}$ with compact support. But by our definition of
$\mathcal{Q}$, $\int z\hat{\pi}<\infty$. So by the dominated
convergence theorem,
$\int\widehat{\mathcal{A}}z\hat{\pi}=0$.
The nontrivial part is to prove that if for any $g\in\mathcal{G}=\{
g(z)+L, g$ is $C_{1}$ with compact support$\}$
such that $\int\widehat{\mathcal{A}}g\hat{\pi}=0$, then $(\hat{\lambda},\hat
{\pi})\in\mathcal{Q}_{e}$. We can easily check the conditions
in Echevrr\'{i}a \cite{Echeverria}. (E.g., $\mathcal{G}$ is
dense in $C(\mathbb{R}^{+})$, the set of continuous and bounded functions
on $\mathbb{R}^{+}$ with limit that exists at infinity and $\widehat{\mathcal{A}}$ satisfies the minimum principle,
that is, $\widehat{\mathcal{A}}f(z_{0})\geq0$ for any $f(z_{0})=\inf_{z\in
\mathbb{R}^{+}}f(z)$. This is because at minimum, the first derivative
of $f$
vanishes and $\hat{\lambda}(z_{0})(f(z_{0}+a)-f(z_{0}))\geq0$. The
other conditions in Echeverr\'{i}a \cite{Echeverria} can also
be easily verified.) Thus, Echevrr\'{i}a \cite{Echeverria} implies that
$\hat{\pi}$ is an
invariant measure. Now, our proof in Lemma \ref{ergodiclemma} shows
that $\hat{\pi}$ has to be unique as well.
Therefore, $(\hat{\lambda},\hat{\pi})\in\mathcal{Q}_{e}$. This implies
that if $(\hat{\lambda},\hat{\pi})\in\mathcal{Q}\setminus\mathcal{Q}_{e}$,
there exists some $g\in\mathcal{G}$, such that $\int\widehat{\mathcal{A}}g\hat{\pi}\neq0$. Now any constant multiplier of $g$ still belongs
to $\mathcal{G}$ and thus $\inf_{g\in\mathcal{G}}\int\widehat{\mathcal
{A}}g\hat{\pi}=-\infty$ and
hence $\inf_{f\in\mathcal{F}}\int\widehat{\mathcal{A}}f\hat{\pi}=-\infty$
if $(\hat{\lambda},\hat{\pi})\in\mathcal{Q}\setminus\mathcal{Q}_{e}$.

Therefore,
%
%e3.23 #&#
%e3.24 #&#
\begin{eqnarray}
\liminf_{t\rightarrow\infty}\frac{1}{t}\log\mathbb{E}
\bigl[e^{\theta N_{t}}\bigr] &\geq&\sup_{(\hat{\lambda},\hat{\pi})\in
\mathcal{Q}}\inf
_{f\in\mathcal{F}} \biggl\{\int\frac{\theta b}{a}z\hat{\pi}-\widehat{H}(\hat{
\lambda},\hat{\pi}) +\int\widehat{\mathcal{A}}f\hat{\pi} \biggr\}\hspace*{-30pt}
\\
\label{switch}
&\geq&\sup_{(\hat{\lambda}\hat{\pi},\hat{\pi})\in\mathcal{R}}\inf_{f\in
\mathcal{F}} \biggl\{\int
\frac{\theta b}{a}z\hat{\pi} -\widehat{H}(\hat{\lambda},\hat{\pi})+\int\widehat{
\mathcal{A}}f\hat{\pi} \biggr\},\hspace*{-30pt}
\end{eqnarray}
where $\mathcal{R}=\{(\hat{\lambda}\hat{\pi},\hat{\pi})\dvtx (\hat{\lambda
},\hat{\pi})\in\mathcal{Q}\}$ and
%
%e3.25 #&#
\begin{equation}
\widehat{H}(\hat{\lambda},\hat{\pi})=\int\bigl[(\lambda-\hat{\lambda
})+\log(\hat{
\lambda}/\lambda)\hat{\lambda} \bigr]\hat{\pi}.
\end{equation}
Define
%
%e3.26 #&#
\begin{eqnarray}
F(\hat{\lambda}\hat{\pi},\hat{\pi},f)&=&\int\frac{\theta b}{a}z\hat{\pi
}-\widehat{H}(
\hat{\lambda},\hat{\pi})+\int\widehat{\mathcal{A}}f\hat{\pi}\nonumber
\\
&=&\int\frac{\theta b}{a}z\hat{\pi}-\widehat{H}(\hat{\lambda},\hat{\pi
})-\int bz
\frac{\partial f}{\partial z}\hat{\pi}
\\
&&{} +\int\bigl(f(z+a)-f(z)\bigr)\hat
{\lambda}\hat{\pi}.
\nonumber
\end{eqnarray}
Notice that $F$ is linear in $f$ and hence convex in $f$ and also
%
%e3.27 #&#
\begin{equation}
\widehat{H}(\hat{\lambda},\hat{\pi})=\sup_{f\in C_{b}(\mathbb{R}^{+})}
\biggl\{\int
\bigl[\hat{\lambda}f+\lambda\bigl(1-e^{f}\bigr) \bigr]\hat{\pi} \biggr
\},
\end{equation}
where $C_{b}(\mathbb{R}^{+})$ denotes the set of bounded functions on
$\mathbb{R}^{+}$. Inside the bracket above,
it is linear in both $\hat{\pi}$ and $\hat{\lambda}\hat{\pi}$. Hence
$\widehat{H}$ is weakly lower semicontinuous and convex in
$(\hat{\lambda}\hat{\pi},\hat{\pi})$. Therefore, $F$ is concave in
$(\hat{\lambda}\hat{\pi},\hat{\pi})$. Furthermore, for any $f=Kz+g+L\in
\mathcal{F}$,
%
%e3.28 #&#
\begin{eqnarray}
F(\hat{\lambda}\hat{\pi},\hat{\pi},f)&=&\int\biggl(\frac{\theta}{a}-K
\biggr)bz\hat{\pi}-\widehat{H}(\hat{\lambda},\hat{\pi}) -\int bz\frac
{\partial g}{\partial z}
\hat{\pi}
\nonumber\\[-8pt]\\[-8pt]
&&{}+\int\bigl(g(z+a)-g(z)\bigr)\hat{\lambda}\hat{\pi}+Ka\int\hat{\lambda
}\hat{
\pi}.
\nonumber
\end{eqnarray}
If $\lambda_{n}\pi_{n}\rightarrow\gamma_{\infty}$ and $\pi
_{n}\rightarrow\pi_{\infty}$ weakly, then, since $g$ is $C_{1}$ with
compact support, we have
%
%e3.29 #&#
\begin{eqnarray}
&&-\int bz\frac{\partial g}{\partial z}\pi_{n}+\int\bigl(g(z+a)-g(z)\bigr)
\lambda_{n}\pi_{n}+Ka\int\lambda_{n}
\pi_{n}
\nonumber\\[-8pt]\\[-8pt]
&&\qquad \rightarrow-\int bz\frac{\partial g}{\partial z}\pi_{\infty}+\int\bigl
(g(z+a)-g(z)
\bigr)\gamma_{\infty}+Ka\int\gamma_{\infty},
\nonumber
\end{eqnarray}
as $n\rightarrow\infty$. Moreover, in general, if $P_{n}\rightarrow P$
weakly, then, for any $f$ which is upper semicontinuous and
bounded from above, we have $\limsup_{n}\int f\,dP_{n}\leq\int f\,dP$.
Since $ (\frac{\theta}{a}-K )bz$ is continuous and nonpositive
on $\mathbb{R}^{+}$, we have
%
%e3.30 #&#
\begin{equation}
\limsup_{n\rightarrow\infty}\int\biggl(\frac{\theta}{a}-K \biggr)bz
\pi_{n}\leq\int\biggl(\frac{\theta}{a}-K \biggr)bz\pi_{\infty}.
\end{equation}
Hence, we conclude that $F$ is upper semicontinuous in the weak topology.

In order to switch the supremum and infimum in \eqref{switch}, since we
have already proved that $F$ is concave,
upper semicontinuous in $(\hat{\lambda}\hat{\pi},\hat{\pi})$ and convex
in $f$, it is sufficient to prove the compactness of $\mathcal{R}$
to apply Ky Fan's minmax theorem; see Fan \cite{Fan}. Indeed, Jo\'{o}
developed some level set method
and proved that it is sufficient to show the compactness of the level
set; see Jo\'{o} \cite{Joo} and Frenk and Kassay \cite{Frenk}.
In other words, it suffices to prove that, for any $C\in\mathbb{R}$ and
$f\in\mathcal{F}$, the level set
%
%e3.31 #&#
\begin{equation}
\qquad\quad \biggl\{(\hat{\lambda}\hat{\pi},\hat{\pi})\in\mathcal{R}\dvtx  \widehat{H}+\int bz
\frac{\partial f}{\partial z}\hat{\pi} -\frac{\theta b}{a}z\hat{\pi
}-\hat{\lambda}
\bigl[f(z+a)-f(z)\bigr]\hat{\pi}\leq C \biggr\}\label{pairinlevelset}
\end{equation}
is compact.\vadjust{\goodbreak}

Fix any $f=Kz+g+L\in\mathcal{F}$, where $K>\frac{\theta}{a}$ and $g$ is
$C_{1}$ with compact support and $L$ is some constant, uniformly for any
pair $(\hat{\lambda}\hat{\pi},\hat{\pi})$ that is in the level set of~\eqref{pairinlevelset}, there exists some $C_{1},C_{2}>0$ such that
%
%e3.32 #&#
\begin{eqnarray}
C_{1}&\geq&\widehat{H}+ \biggl(K-\frac{\theta}{a} \biggr)b\int z\hat{
\pi}-C_{2}\int\hat{\lambda}\hat{\pi}\nonumber
\\
&\geq&\int_{\hat{\lambda}\geq cz+\ell} \bigl[\lambda-\hat{\lambda}+\hat
{\lambda}
\log(\hat{\lambda}/\lambda) \bigr]\hat{\pi} + \biggl(K-\frac{\theta}{a}
\biggr)b
\int z\hat{\pi}
\nonumber
\\
&&{} -C_{2}\int_{\hat{\lambda}\geq cz+\ell}\hat{\lambda}\hat{
\pi}-C_{2}\int_{\hat{\lambda}<cz+\ell}\hat{\lambda}\hat{\pi}
\\
&\geq& \biggl[\min_{z\geq0}\log\frac{cz+\ell}{\lambda(z)}-1-C_{2}
\biggr]\int_{\hat{\lambda}\geq cz+\ell}\hat{\lambda}\hat{\pi} \nonumber
\\
&&{} + \biggl
[-c\cdot
C_{2}+ \biggl(K-\frac{\theta}{a} \biggr)b \biggr]\int z\hat{\pi}-
\ell C_{2}.
\nonumber
\end{eqnarray}
We choose $0<c< (K-\frac{\theta}{a} )\frac{b}{C_{2}}$ and $\ell$ large
enough so that
$\min_{z\geq0}\log\frac{cz+\ell}{\lambda(z)}-1-C_{2}>0$, where we used
the fact that $\lim_{z\rightarrow\infty}\frac{\lambda(z)}{z}=0$
and $\min_{z}\lambda(z)>0$.
Hence,
%
%e3.33 #&#
\begin{equation}
\int z\hat{\pi}\leq C_{3},\qquad\int_{\hat{\lambda}\geq cz+\ell}\hat{
\lambda}\hat{\pi}\leq C_{4},
\end{equation}
where
%
%e3.34 #&#
\begin{eqnarray}
C_{3} &=& \frac{C_{1}+\ell C_{2}}{-c\cdot C_{2}+ (K-(\theta/{a}) )b},
\nonumber\\[-8pt]\\[-8pt]
C_{4} &=&
\frac{C_{1}+\ell C_{2}}{\min_{z\geq0}\log (({cz+\ell})/{\lambda(z)})-1-C_{2}}.\nonumber
\end{eqnarray}
Therefore, we have
%
%e3.35 #&#
\begin{equation}
\int\hat{\lambda}\hat{\pi}=\int_{\hat{\lambda}\geq cz+\ell}\hat{\lambda
}\hat{\pi}+
\int_{\hat{\lambda}<cz+\ell}\hat{\lambda}\hat{\pi} \leq C_{4}+c\cdot
C_{3}+\ell,
\end{equation}
and hence
%
%e3.36 #&#
\begin{equation}
\widehat{H}(\hat{\lambda},\hat{\pi})\leq C_{1}+C_{2}
[C_{4}+c\cdot C_{3}+\ell]<\infty.
\end{equation}

Therefore, for any $(\lambda_{n}\pi_{n},\pi_{n})\in\mathcal{R}$, we get
%
%e3.37 #&#
\begin{equation}
\lim_{\ell\rightarrow\infty}\sup_{n}\int
_{z\geq\ell}\pi_{n}\leq\lim_{\ell\rightarrow\infty}\sup
_{n}\frac{1}{\ell}\int z\pi_{n} \leq\lim
_{\ell\rightarrow\infty}\frac{C_{3}}{\ell}=0,
\end{equation}
which implies the tightness of $\pi_{n}$. By Prokhorov's theorem, there
exists a subsequence of $\pi_{n}$
which converges weakly to $\pi_{\infty}$. We also want to show that
there exists some $\gamma_{\infty}$ such that
$\lambda_{n}\pi_{n}\rightarrow\gamma_{\infty}$ weakly (passing to a
subsequence if necessary).
%Here, $\pi_{n}$ may be a subsequence of the original sequence if
%necessary.
It is enough to show that:
\begin{longlist}[(ii)]
\item[(i)] $\sup_{n}\int\lambda_{n}\pi_{n}<\infty$.

\item[(ii)] $\lim_{\ell\rightarrow\infty}\sup_{n}\int_{z\geq\ell}\lambda_{n}\pi_{n}=0$.
\end{longlist}

(i) and (ii) will give us tightness of $\lambda_{n}\pi_{n}$ and hence
implies the weak convergence for a subsequence.

Now, let us prove statements (i) and (ii).

To prove (i), notice that
%
%e3.38 #&#
\begin{equation}
\sup_{n}\int\lambda_{n}\pi_{n}=\sup
_{n}\int\frac{b}{a}z\pi_{n}\leq
\frac{b}{a}[C_{4}+c\cdot C_{3}+\ell]<\infty.
\end{equation}

To prove (ii), notice that $(\lambda-\lambda_{n})+\lambda_{n}\log
(\lambda_{n}/\lambda)\geq0$. That is because
$x-1-\log x\geq0$ for any $x>0$ and hence
%
%e3.39 #&#
\begin{equation}
\lambda-\hat{\lambda}+\hat{\lambda}\log(\hat{\lambda}/\lambda)=\hat
{\lambda}
\bigl[(\lambda/\hat{\lambda})-1-\log(\lambda/\hat{\lambda}) \bigr]\geq0.
\end{equation}
Notice that
%
%e3.40 #&#
\begin{eqnarray}
&& \lim_{\ell\rightarrow\infty}\sup_{n}\int
_{z\geq\ell}\lambda_{n}\pi_{n}
\nonumber\\[-8pt]\\[-8pt]
&&\qquad \leq \lim
_{\ell\rightarrow\infty} \sup_{n}\int_{\lambda_{n}<\sqrt{\lambda
z},z\geq\ell}
\lambda_{n}\pi_{n}
+\lim_{\ell\rightarrow\infty}\sup_{n}\int
_{\lambda_{n}\geq\sqrt{\lambda z},z\geq\ell}\lambda_{n}\pi_{n}.
\nonumber
\end{eqnarray}
For the first term, since $\sup_{n}\int z\pi_{n}<\infty$ and $\lim
_{z\rightarrow\infty}\frac{\lambda(z)}{z}=0$,
%
%e3.41 #&#
\begin{equation}
\lim_{\ell\rightarrow\infty}\sup_{n}\int
_{\lambda_{n}<\sqrt{\lambda z},z\geq\ell}\lambda_{n}\pi_{n} \leq\lim
_{\ell\rightarrow\infty}\sup_{n}\int_{z\geq\ell}
\sqrt{\lambda z}\pi_{n}=0.
\end{equation}
For the second term, since $\limsup_{z\rightarrow\infty}\frac{\lambda(z)}{z}=0$,
%
%e3.42 #&#
\begin{eqnarray}
\qquad &&\lim_{\ell\rightarrow\infty}\sup_{n}\int
_{\lambda_{n}\geq\sqrt{\lambda z},z\geq\ell}\lambda_{n}\pi_{n}
\nonumber\\[-8pt]\\[-8pt]
&&\qquad \leq\lim_{\ell\rightarrow\infty}\sup_{n}\widehat{H}(
\lambda_{n},\pi_{n})\sup_{\lambda_{n}\geq\sqrt{\lambda z}, z\geq\ell}
\frac{\lambda_{n}}{\lambda-\lambda_{n}+\lambda_{n}\log(\lambda
_{n}/\lambda)}=0.
\nonumber
\end{eqnarray}

Therefore, passing to some subsequence if necessary, we have $\lambda
_{n}\pi_{n}\rightarrow\gamma_{\infty}$ and $\pi_{n}\rightarrow\pi
_{\infty}$ weakly.
Since we proved that $F$ is upper semicontinuous in the weak topology,
the level set is compact in the weak topology.
Therefore, we can switch the supremum and infimum in \eqref{switch} and get
%
%e3.43 #&#
%e3.44 #&#
%e3.45 #&#
%e3.46 #&#
%e3.47 #&#
%e3.48 #&#
\begin{eqnarray}
&& \liminf_{t\rightarrow\infty}\frac{1}{t}\log\mathbb{E}
\bigl[e^{\theta N_{t}} \bigr]
\\
&&\qquad \geq\inf_{f\in\mathcal{F}}\sup_{\hat{\pi}\dvtx \int z\hat{\pi}<\infty}\sup
_{\hat{\lambda}\in L^{1}(\hat{\pi})} \biggl\{\int\frac{\theta b}{a}z\hat
{\pi}+(\hat{\lambda}-
\lambda)\hat{\pi}
\nonumber\\[-8pt]\\[-8pt]
&&\hspace*{131pt}{} -\log(\hat{\lambda}/\lambda)\hat{\lambda}\hat{\pi}+\widehat{\mathcal{A}}f\hat{\pi} \biggr\}\nonumber
\\
&&\qquad =\inf_{f\in\mathcal{F}}\sup_{\hat{\pi}\dvtx \int z\hat{\pi}<\infty} \int
\biggl[
\frac{\theta bz}{a}+\lambda(z) \bigl(e^{f(z+a)-f(z)}-1\bigr)-bz\frac
{\partial f}{\partial z}
\biggr]\hat{\pi}(dz)\label{optimallambda}\hspace*{-40pt}
\\
&&\qquad =\inf_{f\in\mathcal{F}}\sup_{z\geq0} \biggl[
\frac{\theta bz}{a}+\lambda(z) \bigl(e^{f(z+a)-f(z)}-1\bigr)-bz\frac
{\partial f}{\partial z}
\biggr]\label{optimalpi}
\\
&&\qquad =\inf_{f\in\mathcal{F}}\sup_{z\geq0} \biggl[
\frac{\theta bze^{f(z)}}{ae^{f(z)}} +\frac{\lambda(z)}{e^{f(z)}}\bigl
(e^{f(z+a)}-e^{f(z)}
\bigr)-\frac{bz}{e^{f(z)}}\frac{\partial e^{f(z)}}{\partial z} \biggr]\hspace*{-40pt}
\\
&&\qquad \geq\inf_{u\in\mathcal{U}_{\theta}}\sup_{z\geq0} \biggl\{
\frac{\mathcal{A}u}{u}+\frac{\theta b}{a}z \biggr\}.\label{fandu}
\end{eqnarray}
We\vspace*{2pt} need some justifications. Define $G(\hat{\lambda})=\hat{\lambda}-\log
(\hat{\lambda}/\lambda)\hat{\lambda}+\widehat{\mathcal{A}}f$.
The supremum of $G(\hat{\lambda})$ is achieved when $\frac{\partial
G}{\partial\hat{\lambda}}=0$
which implies $\hat{\lambda}=\lambda e^{f(z+a)-f(z)}$. Notice that for
$f\in\mathcal{F}$,
the optimal $\hat{\lambda}=\lambda e^{f(z+a)-f(z)}$ satisfies $\int\hat
{\lambda}\hat{\pi}<\infty$
since $\int\lambda\hat{\pi}<\infty$ and $\int z\hat{\pi}<\infty$. This
gives us \eqref{optimallambda}.
Next, let us explain \eqref{optimalpi}. For any probability measure
$\hat{\pi}$,
%
%e3.49 #&#
\begin{eqnarray}
&& \int\biggl[\frac{\theta bz}{a}+\lambda(z) \bigl(e^{f(z+a)-f(z)}-1\bigr)-bz
\frac{\partial f}{\partial z} \biggr]\hat{\pi}(dz)
\nonumber\\[-8pt]\\[-8pt]
&&\qquad \leq\sup_{z\geq0} \biggl[\frac{\theta bz}{a}+\lambda(z)
\bigl(e^{f(z+a)-f(z)}-1\bigr)-bz\frac{\partial f}{\partial z} \biggr],\nonumber
\end{eqnarray}
which implies the right-hand side of \eqref{optimallambda} is less or
equal to the right-hand side of \eqref{optimalpi}. To prove
the other direction.
For any $f=Kz+g+L\in\mathcal{F}$, we have
%
%e3.50 #&#
\begin{eqnarray}
&& \frac{\theta bz}{a}+\lambda(z) \bigl(e^{f(z+a)-f(z)}-1\bigr)-bz
\frac{\partial f}{\partial z}
\nonumber\\[-8pt]\\[-8pt]
&&\qquad = \biggl(\frac{\theta b}{a}-Kb \biggr)z+\lambda(z) \bigl(e^{Ka+g(z+a)-g(z)}-1
\bigr)-bz\frac{\partial g}{\partial z},
\nonumber
\end{eqnarray}
which is continuous in $z$ and also bounded on $z\in[0,\infty)$ since
$g$ is $C^{1}$ with compact support
and $K>\frac{\theta}{a}$ and $\lim_{z\rightarrow\infty}\frac{\lambda
(z)}{z}=0$. Hence
there exists some $z^{\ast}\geq0$ such that
%
%e3.51 #&#
\begin{eqnarray}
&&\frac{\theta bz}{a}+\lambda(z) \bigl(e^{f(z+a)-f(z)}-1\bigr)-bz
\frac{\partial f}{\partial z}
\nonumber\\[-8pt]\\[-8pt]
&&\qquad =\frac{\theta bz^{\ast}}{a}+\lambda\bigl(z^{\ast}\bigr) \bigl
(e^{f(z^{\ast}+a)-f(z^{\ast})}-1
\bigr)-bz^{\ast}\frac{\partial f}{\partial z} \bigg|_{z=z^{\ast}}.
\nonumber
\end{eqnarray}
Take a sequence of probability measures $\hat{\pi}_{n}$ such that it
has probability density function
$n$ if $z\in[z^{\ast}-\frac{1}{2n},z^{\ast}+\frac{1}{2n}]$ and $0$
otherwise. Then, for every $n$, \mbox{$\int z\hat{\pi}_{n}(dz)<\infty$}.
Therefore, we have
%
%e3.52 #&#
\begin{eqnarray}
\qquad && \lim_{n\rightarrow\infty}\int\biggl[\frac{\theta bz}{a}+\lambda(z)
\bigl(e^{f(z+a)-f(z)}-1\bigr)-bz\frac{\partial f}{\partial z} \biggr
]\hat{\pi}_{n}(dz)\nonumber
\\
&&\qquad =\lim_{n\rightarrow\infty}n\int_{z^{\ast}-({1}/({2n}))}^{z^{\ast}+({1}/{(2n)})}
\biggl[\frac{\theta bz}{a}+\lambda(z) \bigl(e^{f(z+a)-f(z)}-1\bigr)-bz
\frac{\partial f}{\partial z} \biggr]\,dz
\nonumber\\[-8pt]\\[-8pt]
&&\qquad =\frac{\theta bz^{\ast}}{a}+\lambda\bigl(z^{\ast}\bigr) \bigl
(e^{f(z^{\ast}+a)-f(z^{\ast})}-1
\bigr)-bz^{\ast}\frac{\partial f}{\partial z} \bigg|_{z=z^{\ast}}
\nonumber
\\
&&\qquad =\sup_{z\geq0} \biggl[\frac{\theta bz}{a}+\lambda(z)
\bigl(e^{f(z+a)-f(z)}-1\bigr)-bz\frac{\partial f}{\partial z} \biggr].
\nonumber
\end{eqnarray}
We conclude that the right-hand side of \eqref{optimallambda} is
greater or equal to
the right-hand side of \eqref{optimalpi}.

Notice that for any $f=Kz+g+L\in\mathcal{F}$,
%
%e3.53 #&#
\begin{eqnarray}
&&\frac{\theta bz}{a}+\lambda(z) \bigl(e^{f(z+a)-f(z)}-1\bigr)-bz
\frac{\partial f}{\partial z}
\nonumber\\[-8pt]\\[-8pt]
&&\qquad =\frac{b(\theta-Ka)}{a}z+\lambda(z) \bigl(e^{Ka+g(z+a)-g(z)}-1\bigr)-bz
\frac{\partial g}{\partial z},
\nonumber
\end{eqnarray}
whose supremum is achieved at some finite $z^{\ast}>0$ since $\lim
_{z\rightarrow\infty}\frac{\lambda(z)}{z}=0$,
$K>\frac{\theta}{a}$ and $g\in C^{1}$ with compact support. Hence $\int
z\hat{\pi}<\infty$ is satisified for the optimal $\hat{\pi}$.
This gives us \eqref{optimalpi}. Finally, for any $f\in\mathcal{F}$,
$u=e^{f}\in\mathcal{U}_{\theta}$, which implies \eqref{fandu}.
\end{pf}

%le3 #&#
\begin{lemma}\label{thetafinite}
Assume $\lim_{z\rightarrow\infty}\frac{\lambda(z)}{z}=0$, and we have
$\mathbb{E}[e^{\theta N_{t}}]<\infty$
for any \mbox{$\theta\in\mathbb{R}$}.
\end{lemma}

\begin{pf}
Observe that for any $\gamma\in\mathbb{R}$,
%
%e3.54 #&#
\begin{equation}
\exp\biggl\{\gamma N_{t}-\int_{0}^{t}
\bigl(e^{\gamma}-1\bigr)\lambda(Z_{s})\,ds \biggr\}
\end{equation}
is a martinagle. Since $\lim_{z\rightarrow\infty}\frac{\lambda
(z)}{z}=0$, for any $\varepsilon>0$,
there exists a\vspace*{2pt} constant \mbox{$C_{\varepsilon}>0$} such that $\lambda(z)\leq
C_{\varepsilon}+\varepsilon z$ for any $z\geq0$.
Also,
%
%e3.55 #&#
\begin{eqnarray}
\qquad \int_{0}^{t}Z_{s}\,ds&=&\int
_{0}^{t}\!\int_{0}^{s}h(s-u)N(du)\,ds
=\int_{0}^{t} \biggl[\int_{u}^{t}h(s-u)\,ds
\biggr]N(du)
\nonumber\\[-8pt]\\[-8pt]
&\leq&\int_{0}^{t} \biggl[\int
_{u}^{\infty}h(s-u)\,ds \biggr]N(du)= \Vert h \Vert_{L^{1}}N_{t}.\nonumber
\end{eqnarray}
Therefore, for any $\gamma>0$,
%
%e3.56 #&#
\begin{eqnarray}
1&=&\mathbb{E} \bigl[e^{\gamma N_{t}-\int_{0}^{t}(e^{\gamma}-1)\lambda
(Z_{s})\,ds} \bigr]\nonumber
\\
&\geq&\mathbb{E} \bigl[e^{\gamma N_{t}-(e^{\gamma}-1)\int
_{0}^{t}(C_{\varepsilon}+\varepsilon Z_{s})\,ds} \bigr]
\\
&\geq&\mathbb{E} \bigl[e^{\gamma N_{t}-(e^{\gamma}-1)C_{\varepsilon
}t-(e^{\gamma}-1)\varepsilon\Vert h\Vert_{L^{1}}N_{t}} \bigr].
\nonumber
\end{eqnarray}
For any $\theta>0$, choose $\gamma>\theta$ and $\varepsilon$ small enough
so that $\gamma-(e^{\gamma}-1)\varepsilon\Vert h\Vert_{L^{1}}\geq\theta$.
Then
%
%e3.57 #&#
\begin{equation}
\mathbb{E} \bigl[e^{\theta N_{t}} \bigr] \leq e^{(e^{\gamma
}-1)C_{\varepsilon}t}<\infty.
\end{equation}\upqed
\end{pf}

Now we are ready to prove the large deviations result.

%th4 #&#
\begin{theorem}
Assume $\lim_{z\rightarrow\infty}\frac{\lambda(z)}{z}=0$ and that
$\lambda(\cdot)$ is continuous and bounded below by some positive constant.
Then $(\frac{N_{t}}{t}\in\cdot)$ satisfies the large deviation
principle with the rate function $I(\cdot)$ as the Fenchel--Legendre
transform of~$\Gamma(\cdot)$,
%
%e3.58 #&#
\begin{equation}
I(x)=\sup_{\theta\in\mathbb{R}} \bigl\{\theta x-\Gamma(\theta) \bigr\}.
\end{equation}
\end{theorem}

\begin{pf}
If $\limsup_{z\rightarrow\infty}\frac{\lambda(z)}{z}=0$, then the
forthcoming Lemma \ref{Gammafinite} implies that $\Gamma(\theta)<\infty
$ for any $\theta$.
Thus, by the G\"{a}rtner--Ellis theorem, we have the upper bound. For
the G\"{a}rtner--Ellis theorem and a general theory of large deviations,
see, for example, \cite{Dembo}. To prove the lower bound, it suffices
to show that for any $x>0$, $\varepsilon>0$, we have
%
%e3.59 #&#
\begin{equation}
\liminf_{t\rightarrow\infty}\frac{1}{t}\log\mathbb{P} \biggl(
\frac{N_{t}}{t}\in B_{\varepsilon}(x) \biggr)\geq-\sup_{\theta}
\bigl\{\theta x-\Gamma(\theta)\bigr\},
\end{equation}
where $B_{\varepsilon}(x)$ denotes the open ball centered at $x$ with
radius $\varepsilon$. Let $\widehat{\mathbb{P}}$ denote the
tilted~probability measure with rate $\hat{\lambda}$ defined in Theorem
\ref{mainlimit}. By Jensen's inequality,
%
%e3.60 #&#
\begin{eqnarray}
&& \frac{1}{t}\log\mathbb{P} \biggl(\frac{N_{t}}{t}\in
B_{\varepsilon}(x) \biggr)\nonumber
\\
&&\qquad =\frac{1}{t}\log\int_{(N_{t}/t)\in B_{\varepsilon}(x)}\frac{d\mathbb
{P}}{d\widehat{\mathbb{P}}}\,d\widehat{\mathbb{P}}
\\
&&\qquad =\frac{1}{t}\log\widehat{\mathbb{P}} \biggl(\frac{N_{t}}{t}\in
B_{\varepsilon}(x) \biggr)\nonumber
\\
&&\quad\qquad{}  +\frac{1}{t}\log\biggl[\frac{1}{\widehat{\mathbb{P}} (({N_{t}}/{t})\in B_{\varepsilon}(x) )}
\int_{(N_{t}/t)\in B_{\varepsilon}(x)}\frac{d\mathbb{P}}{d\widehat{\mathbb{P}}}\,d\widehat{\mathbb{P}} \biggr]
\nonumber
\\
&&\qquad \geq\frac{1}{t}\log\widehat{\mathbb{P}} \biggl(\frac{N_{t}}{t}\in
B_{\varepsilon}(x) \biggr)\nonumber
\\
&&\quad\qquad{} -\frac{1}{\widehat{\mathbb{P}} (({N_{t}}/{t})\in
B_{\varepsilon}(x) )} \cdot\frac{1}{t}\widehat{\mathbb{E}} \biggl[1_{(N_{t}/t)\in B_{\varepsilon}(x)}\log\frac{d\widehat{\mathbb{P}}}{d\mathbb{P}} \biggr].
\nonumber
\end{eqnarray}
By the ergodic theorem,
%
%e3.61 #&#
\begin{equation}
\liminf_{t\rightarrow\infty}\frac{1}{t}\log\mathbb{P} \biggl(
\frac{N_{t}}{t}\in B_{\varepsilon}(x) \biggr) \geq-\Lambda(x),
\end{equation}
where
%
%e3.62 #&#
\begin{equation}
\label{Lambdaofx} \Lambda(x)=\inf_{(\hat{\lambda},\hat{\pi})\in\mathcal
{Q}_{e}^{x}} \biggl\{\int(\lambda-
\hat{\lambda})\hat{\pi} +\int\log(\hat{\lambda}/\lambda)\hat{\lambda
}\hat{\pi}
\biggr\}
\end{equation}
and
%
%e3.63 #&#
\begin{equation}
\mathcal{Q}_{e}^{x}= \biggl\{(\hat{\lambda},\hat{\pi})\in
\mathcal{Q}_{e}\dvtx \int\hat{\lambda}(z)\hat{\pi}(dz)=x \biggr\}.
\end{equation}
Notice that
%
%e3.64 #&#
\begin{eqnarray}
\Gamma(\theta)&=&\sup_{(\hat{\lambda},\hat{\pi})\in\mathcal{Q}_{e}}
\biggl\{\int\theta\hat{\lambda}
\hat{\pi}+\int(\hat{\lambda}-\lambda)\hat{\pi} -\int\log(\hat{\lambda}/\lambda)
\hat{\lambda}\hat{\pi} \biggr\}\nonumber
\\
&=&\sup_{x}\sup_{(\hat{\lambda},\hat{\pi})\in\mathcal{Q}_{e}^{x}} \biggl
\{\int\theta
\hat{\lambda}\hat{\pi}+\int(\hat{\lambda}-\lambda)\hat{\pi} -\int\log
(\hat{
\lambda}/\lambda)\hat{\lambda}\hat{\pi} \biggr\}
\nonumber\\[-8pt]\\[-8pt]
&=&\sup_{x}\sup_{(\hat{\lambda},\hat{\pi})\in\mathcal{Q}_{e}^{x}} \biggl
\{\int
\frac{\theta b}{a}z\hat{\pi}(dz)+\int(\hat{\lambda}-\lambda)\hat{\pi}
-\int\log(
\hat{\lambda}/\lambda)\hat{\lambda}\hat{\pi} \biggr\}
\nonumber
\\
&=&\sup_{x}\bigl\{\theta x-\Lambda(x)\bigr\}.
\nonumber
\end{eqnarray}
We prove in Lemma \ref{convexity} that $\Lambda(x)$ is convex in $x$,
identify it as the convex conjugate of $\Gamma(\theta)$
and thus complete the proof.
\end{pf}

%le5 #&#
\begin{lemma}\label{convexity}
$\Lambda(x)$ in \eqref{Lambdaofx} is convex in $x$.
\end{lemma}

\begin{pf}
Define
%
%e3.65 #&#
\begin{equation}
\widehat{H}(\hat{\lambda},\hat{\pi})=\int(\lambda-\hat{\lambda})\hat{\pi}+\int\log(
\hat{\lambda}/\lambda)\hat{\lambda}\hat{\pi}.
\end{equation}
Then
%
%e3.66 #&#
\begin{equation}
\Lambda(x)=\inf_{(\hat{\lambda},\hat{\pi})\in\mathcal{Q}_{e}^{x}}\widehat{H}(\hat{\lambda},\hat{\pi}).
\end{equation}
We\vspace*{2pt} want to prove that $\Lambda(\alpha x_{1}+\beta x_{2})\leq\alpha
\Lambda(x_{1})+\beta\Lambda(x_{2})$
for any $\alpha,\beta\geq0$ with $\alpha+\beta=1$. For any $\varepsilon>0$,
we can choose $(\hat{\lambda}_{k},\hat{\pi}_{k})\in\mathcal
{Q}_{e}^{x_{k}}$ such that
$\widehat{H}(\hat{\lambda}_{k},\hat{\pi}_{k})\leq\Lambda(x_{k})+\varepsilon
/2$, for $k=1,2$. Set
%
%e3.67 #&#
\begin{equation}
\qquad \hat{\pi}_{3}=\alpha\hat{\pi}_{1}+\beta\hat{
\pi}_{2}, \qquad\hat{\lambda}_{3}=\frac{d(\alpha\hat{\pi}_{1})}{d(\alpha
\hat{\pi}_{1}+\beta\hat{\pi}_{2})}\hat{
\lambda}_{1}+\frac{d(\beta\hat{\pi}_{2})}{d(\alpha\hat{\pi}_{1}+\beta
\hat{\pi}_{2})}\hat{\lambda}_{2}.
\end{equation}
Then for any test function $f$,
%
%e3.68 #&#
\begin{equation}
\int\widehat{\mathcal{A}}_{3}f\hat{\pi}_{3}=\alpha\int\widehat{\mathcal{A}}_{1}f\hat{\pi}_{1}+\beta\int\widehat{\mathcal{A}}_{2}f\hat{\pi}_{2}=0,
\end{equation}
which implies $(\hat{\lambda}_{3},\hat{\pi}_{3})\in\mathcal{Q}_{e}$.
Furthermore,
%e3.69 #&#
\begin{equation}
\int\hat{\lambda}_{3}\hat{\pi}_{3}=\alpha\int\hat{
\lambda}_{1}\hat{\pi}_{1}+\beta\int\hat{\lambda}_{2}
\hat{\pi}_{2}=\alpha x_{1}+\beta x_{2}.
\end{equation}
Therefore, $(\hat{\lambda}_{3},\hat{\pi}_{3})\in\mathcal{Q}_{e}^{\alpha
x_{1}+\beta x_{2}}$.
Finally, since $x\log x$ is a convex function and if we apply Jensen's
inequality, we get
%
%e3.70 #&#
\begin{eqnarray}
\widehat{H}(\hat{\lambda}_{3},\hat{\pi}_{3}) &=&\int\bigl[(
\lambda-\hat{\lambda}_{3}-\hat{\lambda}_{3}\log\lambda)+
\hat{\lambda}_{3}\log\hat{\lambda}_{3} \bigr]\hat{
\pi}_{3}\nonumber
\\
&\leq&\int\biggl[(\lambda-\hat{\lambda}_{3}-\hat{
\lambda}_{3}\log\lambda) +\alpha\frac{d\hat{\pi}_{1}}{d\hat{\pi
}_{3}}\hat{
\lambda}_{1}\log\hat{\lambda}_{1} +\beta\frac{d\hat{\pi}_{2}}{d\hat{\pi}_{3}}
\hat{\lambda}_{2}\log\hat{\lambda}_{2} \biggr]\hat{
\pi}_{3}\hspace*{-25pt}
\\
&=&\alpha\widehat{H}(\hat{\lambda}_{1},\hat{\pi}_{1})+\beta
\widehat{H}(\hat{\lambda}_{2},\hat{\pi}_{2}).
\nonumber
\end{eqnarray}
Therefore,
%
%e3.71 #&#
\begin{eqnarray}
\Lambda(\alpha x_{1}+\beta x_{2})&\leq&\widehat{H}(\hat{
\lambda}_{3},\hat{\pi}_{3}) \nonumber
\\
&\leq&\alpha\widehat{H}(\hat{
\lambda}_{1},\hat{\pi}_{1})+\beta\widehat{H}(\hat{
\lambda}_{2},\hat{\pi}_{2})
\\
&\leq&\alpha\Lambda(x_{1})+
\beta\Lambda(x_{2})+\varepsilon.\nonumber
\end{eqnarray}\upqed
\end{pf}

%le6 #&#
\begin{lemma}\label{Gammafinite}
If $\limsup_{z\rightarrow\infty}\frac{\lambda(z)}{bz}<\frac{1}{a}$,
then for any
%
%e3.72 #&#
\begin{equation}
\theta<\log\biggl(\frac{b}{a\limsup_{z\rightarrow\infty}({\lambda
(z)}/{z})} \biggr)-1+\frac{a}{b}\cdot\limsup
_{z\rightarrow\infty}\frac{\lambda(z)}{z},
\end{equation}
we have $\Gamma(\theta)<\infty$. If $\limsup_{z\rightarrow\infty}\frac
{\lambda(z)}{z}=0$, then $\Gamma(\theta)<\infty$ for any $\theta\in
\mathbb{R}$.
\end{lemma}

\begin{pf}
For $K\geq\frac{\theta}{a}$, we have $e^{Kz}\in\mathcal{U}_{\theta}$ and
%
%e3.73 #&#
\begin{eqnarray}
\Gamma(\theta)&\leq&\inf_{g\in\mathcal{U}_{\theta}}\sup_{z\geq0}
\frac{\mathcal{A}g(z)
+(({\theta b})/{a})zg(z)}{g(z)}\leq\sup_{z\geq0} \biggl\{\frac{\mathcal
{A}e^{Kz}}{e^{Kz}}+
\frac{\theta b}{a}z \biggr\}
\nonumber\\[-8pt]\\[-8pt]
&=&\sup_{z\geq0} \biggl\{- \biggl(bK-\frac{\theta b}{a} \biggr)z+
\lambda(z) \bigl(e^{Ka}-1\bigr) \biggr\}.
\nonumber
\end{eqnarray}
Define the function
%
%e3.74 #&#
\begin{equation}
F(K)=-K+\limsup_{z\rightarrow\infty}\frac{\lambda(z)}{bz}\cdot
\bigl(e^{Ka}-1\bigr).
\end{equation}
Then $F(0)=0$, $F$ is convex and $F(K)\rightarrow\infty$ as
$K\rightarrow\infty$ and its minimum is attained at
%
%e3.75 #&#
\begin{equation}
K^{\ast}=\frac{1}{a}\log\biggl(\frac{b}{a\limsup_{z\rightarrow\infty
}({\lambda(z)}/{z})} \biggr)>0,
\end{equation}
and $F(K^{\ast})<0$. Therefore, $\Gamma(\theta)<\infty$ for any
%
%e3.76 #&#
\begin{eqnarray}
\theta&<&-a\min_{K>0} \biggl\{-K+\limsup_{z\rightarrow\infty}
\frac{\lambda(z)}{bz}\cdot\bigl(e^{Ka}-1\bigr) \biggr\}
\nonumber\\[-8pt]\\[-8pt]
&=&\log\biggl(\frac{b}{a\limsup_{z\rightarrow\infty}({\lambda
(z)}/{z})} \biggr)-1 +\frac{a}{b}\cdot\limsup
_{z\rightarrow\infty}\frac{\lambda(z)}{z}<K^{\ast}a.
\nonumber
\end{eqnarray}
If $\limsup_{z\rightarrow\infty}\frac{\lambda(z)}{z}=0$, trying
$e^{Kz}\in\mathcal{U}_{\theta}$
for any $K>\frac{\theta}{a}$, we have $\Gamma(\theta)<\infty$ for any
$\theta$.
\end{pf}

%s4 #&#
\section{Large deviations for Markovian nonlinear Hawkes processes with sum of exponentials exciting function}\label{sec4}

In this section, we consider the Markovian nonlinear Hawkes processes
with sum of exponentials exciting functions,
that is, $h(t)=\sum_{i=1}^{d}a_{i}e^{-b_{i}t}$.
Let
%
%e4.1 #&#
\begin{equation}
Z_{i}(t)=\sum_{\tau_{j}<t}a_{i}e^{-b_{i}(t-\tau_{j})},
\qquad1\leq i\leq d
\end{equation}
and $Z_{t}=\sum_{i=1}^{d}Z_{i}(t)=\sum_{\tau_{j}<t}h(t-\tau_{j})$,
where $\tau_{j}$'s are the arrivals of the Hawkes process
with intensity $\lambda(Z_{t})=\lambda(Z_{1}(t)+\cdots+Z_{d}(t))$ at
time $t$. Observe that this is a special
case of the Markovian processes with\vspace*{2pt} jumps studied in Section~\ref
{ErgodicSection} with $\lambda(Z_{1}(t),Z_{2}(t),\ldots,Z_{d}(t))$
taking the form $\lambda(\sum_{i=1}^{d}Z_{i}(t))$.
It is easy to see that $(Z_{1},\ldots,Z_{d})$ is Markovian with generator
%
%e4.2 #&#
\begin{eqnarray}
\mathcal{A}f &=& -\sum_{i=1}^{d}b_{i}z_{i}
\frac{\partial f}{\partial z_{i}}
\nonumber\\[-8pt]\\[-8pt]
&&{} +\lambda\Biggl(\sum_{i=1}^{d}z_{i}
\Biggr)\cdot\bigl[f(z_{1}+a_{1},\ldots,z_{d}+a_{d})-f(z_{1},
\ldots,z_{d})\bigr].\nonumber
\end{eqnarray}
Here $b_{i}>0$ for any $1\leq i\leq d$ and $a_{i}$ can be negative. But
we restrict ourselves
to the set of $b_{i}$'s and $a_{i}$'s so that $h(t)=\sum
_{i=1}^{d}a_{i}e^{-b_{i}t}>0$ for any $t\geq0$
for the rest of this paper.
In particular, $h(0)=\sum_{i=1}^{d}a_{i}>0$. If $a_{i}>0$, then
$Z_{i}(t)\geq0$ almost surely; if $a_{i}<0$, then $Z_{i}(t)\leq0$
almost surely.

%th7 #&#
\begin{theorem}
Assume $\lim_{z\rightarrow\infty}\frac{\lambda(z)}{z}=0$, $\lambda(\cdot
)$ is continuous
and bounded below by a positive constant. Then
%
%e4.3 #&#
\begin{equation}
\lim_{t\rightarrow\infty}\frac{1}{t}\log\mathbb{E}
\bigl[e^{\theta N_{t}}\bigr] =\inf_{u\in\mathcal{U}_{\theta}}\sup
_{(z_{1},\ldots,z_{d})\in\mathcal{Z}} \Biggl\{\frac{\mathcal
{A}u}{u}+\frac{\theta}{\sum_{i=1}^{d}a_{i}} \sum
_{i=1}^{d}b_{i}z_{i}
\Biggr\},
\end{equation}
where $\mathcal{Z}=\{(z_{1},\ldots,z_{d})\dvtx  a_{i}z_{i}\geq0, 1\leq
i\leq d\}$ and
%
%e4.4 #&#
\begin{equation}
\mathcal{U}_{\theta}= \bigl\{u\in C_{1}\bigl(
\mathbb{R}^{d},\mathbb{R}^{+}\bigr),u=e^{f}, f\in
\mathcal{F} \bigr\},
\end{equation}
where
%
%e4.5 #&#
\begin{equation}
\mathcal{F}= \biggl\{f=g+\frac{\theta\sum_{i=1}^{d}z_{i}}{\sum
_{i=1}^{d}a_{i}}+L, L\in\mathbb{R}, g\in\mathcal{G}
\biggr\},
\end{equation}
where
%
%e4.6 #&#
\begin{equation}
\mathcal{G}= \Biggl\{\sum_{i=1}^{d}K
\varepsilon_{i}z_{i}+g, K>0, g\mbox{ is }C_{1}
\mbox{ with compact support} \Biggr\}.
\end{equation}
\end{theorem}

\begin{pf}
Notice that
%
%e4.7 #&#
\begin{equation}
dZ_{i}(t)=-b_{i}Z_{i}(t)\,dt+a_{i}\,dN_{t},
\qquad1\leq i\leq d.
\end{equation}
Hence $a_{i}N_{t}=Z_{i}(t)-Z_{i}(0)+\int_{0}^{t}b_{i}Z_{i}(s)\,ds$ and
%
%e4.8 #&#
\begin{equation}
\mathbb{E}\bigl[e^{\theta N_{t}}\bigr]=\mathbb{E} \Biggl[\exp\Biggl\{
\frac{\theta\sum_{i=1}^{d}Z_{i}(t)-Z_{i}(0)}{\sum_{i=1}^{d}a_{i}} +\frac
{\theta}{\sum_{i=1}^{d}a_{i}}\int_{0}^{t}
\sum_{i=1}^{d}b_{i}Z_{i}(s)\,ds
\Biggr\} \Biggr].\hspace*{-30pt}
\end{equation}
Following the same arguments in the proof of Theorem \ref{mainlimit},
we obtain the upper bound
%
%e4.9 #&#
\begin{equation}
\limsup_{t\rightarrow\infty}\frac{1}{t}\log\mathbb{E}
\bigl[e^{\theta N_{t}}\bigr]\leq\inf_{u\in\mathcal{U}_{\theta}} \sup
_{(z_{1},\ldots,z_{d})\in\mathcal{Z}} \Biggl\{\frac{\mathcal
{A}u}{u}+\frac{\theta}{\sum_{i=1}^{d}a_{i}}\sum
_{i=1}^{d}b_{i}z_{i}
\Biggr\}.
\end{equation}
As before, we can obtain the lower bound
%
%e4.10 #&#
\begin{eqnarray}
&&\liminf_{t\rightarrow\infty}\frac{1}{t}\log\mathbb{E}
\bigl[e^{\theta N_{t}}\bigr]\nonumber
\\
&&\qquad \geq\sup_{(\hat{\lambda},\hat{\pi})\in\mathcal{Q}_{e}}\int\bigl[\theta
\hat{\lambda}-\lambda+\hat{
\lambda} -\hat{\lambda}\log(\hat{\lambda}/\lambda) \bigr]\hat{
\pi}(dz_{1},\ldots,dz_{d})
\nonumber\\[-8pt]\\[-8pt]
&&\qquad \geq\sup_{(\hat{\lambda},\hat{\pi})\in\mathcal{Q}}\inf_{g\in\mathcal
{G}}\int\bigl[\theta
\hat{\lambda}-\lambda+\hat{\lambda}-\hat{\lambda}\log(\hat{\lambda
}/\lambda)+
\widehat{\mathcal{A}}g \bigr]\hat{\pi}
\nonumber
\\
&&\qquad =\sup_{(\hat{\lambda},\hat{\pi})\in\mathcal{Q}}\inf_{f\in\mathcal{F}}
\int\biggl[
\frac{\theta\sum_{i=1}^{d}b_{i}z_{i}}{\sum_{i=1}^{d}a_{i}}-\lambda+\hat
{\lambda}-\hat{\lambda}\log(\hat{\lambda}/\lambda
) +\widehat{\mathcal{A}}f \biggr]\hat{\pi}.
\nonumber
\end{eqnarray}
The equality in the last line above holds by taking $f=g+L+\frac{\theta
\sum_{i=1}^{d}z_{i}}{\sum_{i=1}^{d}a_{i}}\in\mathcal{F}$ for $g\in
\mathcal{G}$, where
%
%e4.11 #&#
\begin{equation}
\mathcal{G}= \Biggl\{\sum_{i=1}^{d}K
\varepsilon_{i}z_{i}+g, K>0, g\mbox{ is }C_{1}
\mbox{ with compact support} \Biggr\}.
\end{equation}
Here, $\varepsilon_{i}=a_{i}/|a_{i}|$, $1\leq i\leq d$. Define
%
%e4.12 #&#
\begin{equation}
F(\hat{\lambda}\hat{\pi},\hat{\pi},f)=\int\biggl[\frac{\theta\sum
_{i=1}^{d}b_{i}z_{i}}{\sum_{i=1}^{d}a_{i}}+\widehat{\mathcal{A}}f \biggr]\hat{\pi} -\widehat{H}(\hat{\lambda},\hat{\pi}).
\end{equation}
$F$ is linear in $f$ and hence convex in $f$. Also $\widehat{H}$ is weakly
lower semicontinuous and convex in $(\hat{\lambda}\hat{\pi},\hat{\pi})$.
Therefore, $F$ is concave in $(\hat{\lambda}\hat{\pi},\hat{\pi})$.
Furthermore, for any $f=\frac{\theta\sum_{i=1}^{d}z_{i}}{\sum
_{i=1}^{d}a_{i}}+\sum_{i=1}^{d}K\varepsilon_{i}z_{i}+g+L\in\mathcal{F}$,
%
%e4.13 #&#
\begin{eqnarray}
F(\hat{\lambda}\hat{\pi},\hat{\pi},f)&=&\int\Biggl[\theta+\sum
_{i=1}^{d}K\varepsilon_{i}a_{i}
\Biggr]\hat{\lambda}\hat{\pi}
\nonumber\\[-8pt]\\[-8pt]
&&{} -\int\sum_{i=1}^{d}K
\varepsilon_{i}b_{i}z_{i}\hat{\pi}-\widehat{H}(\hat{
\lambda},\hat{\pi})+\int\widehat{\mathcal{A}}g\hat{\pi}.\nonumber
\end{eqnarray}
If $\lambda_{n}\pi_{n}\rightarrow\gamma_{\infty}$ and $\pi
_{n}\rightarrow\pi_{\infty}$ weakly,
then, since $g$ is $C_{1}$ with compact support, we have
%
%e4.14 #&#
\begin{eqnarray}
&& \int\Biggl[\theta+\sum_{i=1}^{d}K
\varepsilon_{i}a_{i} \Biggr]\lambda_{n}
\pi_{n}+\int\widehat{\mathcal{A}}g\pi_{n}
\nonumber\\[-8pt]\\[-8pt]
&&\qquad \rightarrow\int
\Biggl[\theta+\sum_{i=1}^{d}K
\varepsilon_{i}a_{i} \Biggr]\gamma_{\infty}+\int\widehat{\mathcal{A}}g\pi_{\infty}.\nonumber
\end{eqnarray}
Since $-\sum_{i=1}^{d}K\varepsilon_{i}b_{i}z_{i}$ is continuous and
nonpositive on $\mathcal{Z}$, we have
%
%e4.15 #&#
\begin{equation}
\limsup_{n\rightarrow\infty}\int\Biggl[-\sum_{i=1}^{d}K
\varepsilon_{i}b_{i}z_{i} \Biggr]
\pi_{n} \leq\int\Biggl[-\sum_{i=1}^{d}K
\varepsilon_{i}b_{i}z_{i} \Biggr]
\pi_{\infty}.
\end{equation}
Hence, we conclude that $F$ is upper semicontinuous in the weak topology.

In order to apply the minmax theorem, we want to prove the compactness
in the weak topology of the level set
%
%e4.16 #&#
\begin{equation}
\biggl\{(\hat{\lambda}\hat{\pi},\hat{\pi})\dvtx \int\biggl[-\frac{\theta\sum
_{i=1}^{d}b_{i}z_{i}}{\sum_{i=1}^{d}a_{i}} -
\widehat{\mathcal{A}}f \biggr]\hat{\pi}+\widehat{H}(\hat{\lambda},\hat{\pi
})\leq C \biggr
\}.
\end{equation}
For any $f=\frac{\theta\sum_{i=1}^{d}z_{i}}{\sum_{i=1}^{d}a_{i}}+\sum
_{i=1}^{d}K\varepsilon_{i}z_{i}+g+L\in\mathcal{F}$,
where $g$ is $C_{1}$ with compact support, etc., there exist some
$C_{1}, C_{2}>0$ such that
%
%e4.17 #&#
\begin{eqnarray}
\qquad C_{1}&\geq&\widehat{H}+\sum_{i=1}^{d}Kb_{i}
\varepsilon_{i}\int z_{i}\hat{\pi}-C_{2}\int\hat{
\lambda}\hat{\pi}\nonumber
\\
&\geq&\int_{\hat{\lambda}\geq\sum_{i=1}^{d}c_{i}z_{i}+\ell} \bigl[\lambda
-\hat{\lambda} +\hat{\lambda}
\log(\hat{\lambda}/\lambda) \bigr]\hat{\pi}\nonumber\\
&&{}+\sum_{i=1}^{d}Kb_{i}
\varepsilon_{i}\int z_{i}\hat{\pi}
\nonumber
\\[-8pt]
\\[-8pt]
\nonumber
&&{}-C_{2}\int_{\hat{\lambda}\geq\sum_{i=1}^{d}c_{i}z_{i}+\ell}\hat{\lambda
}\hat{
\pi}-C_{2}\int_{\hat{\lambda}
<\sum_{i=1}^{d}c_{i}z_{i}+\ell}\hat{\lambda}\hat{\pi}
\\
&\geq&\biggl[\min_{(z_{1},\ldots,z_{d})\in\mathcal{Z}}\log\frac
{c_{1}z_{1}+\cdots+c_{d}z_{d}+\ell}{\lambda(z_{1}+\cdots+z_{d})}-1-C_{2}
\biggr] \int_{\hat{\lambda}\geq\sum_{i=1}^{d}c_{i}z_{i}+\ell}\hat
{\lambda}\hat{\pi}
\nonumber
\\
&&{}+\sum_{i=1}^{d}[-c_{i}\cdot
C_{2}+Kb_{i}\varepsilon_{i}]\int
z_{i}\hat{\pi}-\ell C_{2}.
\nonumber
\end{eqnarray}
If $a_{i}>0$, then $\varepsilon_{i}>0$, pick up $c_{i}>0$ such that
$-c_{i}\cdot C_{2}+Kb_{i}\varepsilon_{i}>0$.
If $a_{i}<0$, then $\varepsilon_{i}<0$, pick up $c_{i}$ such that
$-c_{i}\cdot C_{2}+Kb_{i}\varepsilon_{i}<0$.
Finally, choose $\ell$ big enough such that the big bracket above is
positive. Then
%
%e4.18 #&#
\begin{equation}
\int|z_{i}|\hat{\pi}\leq C_{3},\qquad\int
_{\hat{\lambda}\geq\sum_{i=1}^{d}c_{i}z_{i}+\ell}\hat{\lambda}\hat{\pi
}\leq C_{4}.
\end{equation}
Hence, $\int\hat{\lambda}\hat{\pi}\leq C_{5}$ and $\widehat{H}\leq C_{6}$.
We can use a method similar to the proof of Theorem \ref{mainlimit} to
show that
%
%e4.19 #&#
\begin{equation}
\lim_{\ell\rightarrow\infty}\sup_{n}\int
_{|z_{i}|>\ell}\lambda_{n}\pi_{n}=0,\qquad1\leq
i\leq d.
\end{equation}
For any $(\lambda_{n}\pi_{n},\pi_{n})\in\mathcal{R}$, we can find a
subsequence that converges in the weak topology by Prokhorov's theorem.
Therefore,
%
%e4.20 #&#
\begin{eqnarray}
\qquad &&\liminf_{t\rightarrow\infty}\frac{1}{t}\log\mathbb{E}
\bigl[e^{\theta N_{t}}\bigr]\nonumber
\\
&&\qquad \geq\sup_{(\hat{\lambda},\hat{\pi})\in\mathcal{Q}}\inf_{f\in\mathcal
{F}} \int\biggl[
\frac{\theta\sum_{i=1}^{d}b_{i}z_{i}}{\sum_{i=1}^{d}a_{i}}-\lambda+\hat
{\lambda}-\hat{\lambda} \log(\hat{\lambda}/\lambda
)+\widehat{\mathcal{A}}f \biggr]\hat{\pi}
\nonumber
\\
&&\qquad =\inf_{f\in\mathcal{F}}\sup_{\hat{\pi}}\sup
_{\hat{\lambda}}\int\biggl[\frac{\theta\sum_{i=1}^{d}b_{i}z_{i}}{\sum
_{i=1}^{d}a_{i}}-\lambda+\hat{\lambda}-
\hat{\lambda}\log(\hat{\lambda}/\lambda)+\widehat{\mathcal{A}}f \biggr
]\hat{\pi}
\nonumber\\
&&\qquad =\inf_{f\in\mathcal{F}}\sup_{(z_{1},\ldots,z_{d})\in\mathcal{Z}}\frac
{\theta\sum_{i=1}^{d}b_{i}z_{i}}{\sum_{i=1}^{d}a_{i}} +
\lambda\bigl(e^{f(z_{1}+a_{1},\ldots,z_{d}+a_{d})-f(z_{1},\ldots,z_{d})}-1\bigr)
\\
&&\quad\qquad{} -\sum_{i=1}^{d}b_{i}z_{i}
\frac{\partial f}{\partial z_{i}}
\nonumber\\
&&\qquad \geq\inf_{u\in\mathcal{U}_{\theta}}\sup_{(z_{1},\ldots,z_{d})\in
\mathcal{Z}} \Biggl\{
\frac{\mathcal{A}u}{u} +\frac{\theta}{\sum_{i=1}^{d}a_{i}}\sum
_{i=1}^{d}b_{i}z_{i}
\Biggr\}.
\nonumber
\end{eqnarray}
That is because optimizing over $\hat{\lambda}$, we get $\hat{\lambda
}=\lambda e^{f(z_{1}+a_{1},\ldots,z_{d}+a_{d})-f(z_{1},\ldots,z_{d})}$
and finally for each $f\in\mathcal{F}$, $u=e^{f}\in\mathcal{U}_{\theta}$.
\end{pf}

%th8 #&#
\begin{theorem}
Assume $\lim_{z\rightarrow\infty}\frac{\lambda(z)}{z}=0$, $\lambda(\cdot
)$ is positive and bounded below by some positive constant.
Then, $(\frac{N_{t}}{t}\in\cdot)$ satisfies the large deviation
principle with the rate function $I(\cdot)$ as the
Fenchel--Legendre transform of $\Gamma(\cdot)$,
%
%e4.21 #&#
\begin{equation}
I(x)=\sup_{\theta\in\mathbb{R}} \bigl\{\theta x-\Gamma(\theta) \bigr\},
\end{equation}
where
%
%e4.22 #&#
\begin{equation}
\Gamma(\theta)=\sup_{(\hat{\lambda},\hat{\pi})\in\mathcal{Q}_{e}}\int
\bigl[\theta\hat{\lambda}-
\lambda+\hat{\lambda} -\hat{\lambda}\log(\hat{\lambda}/\lambda) \bigr
]\hat{\pi}.
\end{equation}
\end{theorem}

\begin{pf}
The proof is the same as in the case of exponential $h(\cdot)$.
\end{pf}

%s5 #&#
\section{Large deviations for linear Hawkes processes: An alternative proof}\label{sec5}

In this section, we use our method to recover the result proved in
Bordenave and Torrisi \cite{Bordenave}.
We prove the existence of the limit of logarithmic moment generating
function first.
The strategy is to use the tilting method to prove the lower bound.
This requires an ergodic lemma, which we state as Lemma \ref{ergodicformula}.
For the upper bound, we can opitimize over a special class of testing
functions for the linear rate with the sum
of exponential exciting function $h_{n}$. Any continuous and integrable
$h$ can be approximated by a sequence $h_{n}$.
By a coupling argument, we can use that to approximate the upper bound
for the logarithmic moment generating function
when the exciting function is $h$. Finally, by a tilting argument for
the lower bound and the G\"{a}rtner--Ellis theorem for the upper bound,
we can prove the large deviations for the linear Hawkes processes.

%le9 #&#
\begin{lemma}\label{ergodicformula}
Assume $\lambda(z)=\alpha+\beta z$ and $\mu=\int_{0}^{\infty
}h(t)\,dt<\infty$. If $\beta\mu<1$,
then there exists a stationary and ergodic probability measure $\pi$
for $Z_{t}$ and $\int z\pi=\frac{\alpha\mu}{1-\beta\mu}$.
\end{lemma}

\begin{pf}
The ergodicity is a well-known result for linear Hawkes process; see
Hawkes and Oakes \cite{HawkesII}.
Let $\pi$ be the invariant probability measure for $Z_{t}$, then
%
%e5.1 #&#
\begin{equation}
\lim_{t\rightarrow\infty}\frac{N_{t}}{t}=\int\lambda(z)\pi(dz)=\alpha+
\beta\int z\pi(dz).
\end{equation}
If $Z_{t}$ is invariant in $t$, taking expectations to $Z_{t}=\int
_{-\infty}^{t}h(t-s)\,dN_{s}$,
%
%e5.2 #&#
\begin{eqnarray}
 \mathbb{E}[Z_{t}]&=&\int z\pi(dz)=\int\lambda(z)\pi(dz)\int
_{-\infty}^{t}h(t-s)\,ds
\nonumber
\\[-8pt]
\\[-8pt]
\nonumber
&=&\mu\int\lambda(z)\pi(dz),
\end{eqnarray}
which implies that $\int z\pi=\frac{\alpha\mu}{1-\beta\mu}$.
\end{pf}

%re10 #&#
\begin{remark}
In Lemma \ref{ergodicformula}, we assumed that $\lambda(z)=\alpha+\beta
z$ and $\beta\Vert h\Vert_{L^{1}}<1$. However, when do the LDP
for linear Hawkes process and when we prove Theorem~\ref{linearcase},
we assume that $\lambda(z)=\nu+z$ since $\lambda(z)=\nu+\beta z$ is
equivalent to the case $\lambda(z)=\nu+z$ if we change $h(\cdot)$ to
$\beta h(\cdot)$. The reason we used $\lambda(z)=\alpha+\beta z$
in Lemma \ref{ergodicformula} is because we need to use it when we tilt
$\lambda(z)=\nu+z$ to $K\lambda(z)=K\nu+Kz$ in the proof of lower bound
in Theorem \ref{linearcase}.
\end{remark}

%le11 #&#
\begin{lemma}\label{littleh}
If $h(t)>0$, $\int_{0}^{\infty}h(t)\,dt<\infty$, $\lim_{t\rightarrow\infty
}h(t)=0$, and $h$ is continuous,
then $h$ can be approximated by a sum of exponentials both in $L^{1}$
and $L^{\infty}$ norms.
\end{lemma}

\begin{pf}
The Stone--Weierstrass theorem says that if $X$ is a compact Hausdorff
space and suppose $A$ is a subspace of $C(X)$
with the following properties: (i)~If $f,g\in A$, then $f\times g\in
A$. (ii) $1\in A$.
(iii) If $x,y\in X$, then we can find an $f\in A$ such that $f(x)\neq
f(y)$, then $A$ is dense in $C(X)$ in $L^{\infty}$ norm.
Consider $X=\mathbb{R}_{\geq0}\cup\{\infty\}=[0,\infty]$ that is compactified
and $C[0,\infty]$ consists of continuous functions vanishing at $\infty
$ and the constant function~$1$.

By the Stone--Weierstrass theorem, the linear combination of $1$,
$e^{-t}$, $e^{-2t}$, etc., is dense in $C[0,\infty]$.
In other words, for any continuous function $h$ on $C[0,\infty]$, we have
%
%e5.3 #&#
\begin{equation}
\sup_{t\geq0}\Biggl\llvert h(t)-\sum
_{j=0}^{n}a_{j}e^{-jt}\Biggr
\rrvert\leq\varepsilon.
\end{equation}
In fact, since $h(\infty)=0$, we get $|a_{0}|\leq\varepsilon$. Thus
%
%e5.4 #&#
\begin{equation}
\sup_{t\geq0}\Biggl\llvert h(t)-\sum
_{j=1}^{n}a_{j}e^{-jt}\Biggr
\rrvert\leq2\varepsilon.
\end{equation}
However, $\sum_{j=1}^{n}a_{j}e^{-jt}$ may not be positive. We can
approximate $\sqrt{h(t)}$ first by a sum of exponentials
and then approximate $h(t)$ by the square of that sum of exponentials,
which is again a sum of exponentials but positive this time.

Indeed, we can approximate $h(t)$ by the sum of exponentials in $L^{1}$
norm as well.
Suppose $\Vert h-h_{n}\Vert_{L^{\infty}}\rightarrow0$, where $h_{n}$
is a\vspace*{2pt} sum of exponentials.
Then, by dominated convergence theorem, for any $\delta>0$, $\int
|h-h_{n}|e^{-\delta t}\,dt\rightarrow0$ as $n\rightarrow\infty$.
Thus, we can find a sequence $\delta_{n}>0$ such that $\delta
_{n}\rightarrow0$ as $n\rightarrow\infty$ and $\int|h-h_{n}|e^{-\delta
_{n}t}\,dt\rightarrow0$.
By dominated convergence theorem again, $\int h(1-e^{-\delta
_{n}t})\,dt\rightarrow0$.
Hence, we have $\int|h-h_{n}e^{-\delta_{n}t}|\,dt\rightarrow0$ as
$n\rightarrow\infty$, where $h_{n}e^{-\delta_{n}t}$ is a sum of exponentials.

We will show that $h_{n}e^{-\delta_{n}t}$ converges to $h$ in $L^{\infty
}$ as well.
%
%e5.5 #&#
\begin{equation}
\big\Vert h-h_{n}e^{-\delta_{n}t}\big\Vert_{L^{\infty}}\leq\Vert
h-h_{n}\Vert_{L^{\infty}}+\big\Vert h_{n}-h_{n}e^{-\delta_{n}t}
\big\Vert_{L^{\infty}}.
\end{equation}
Notice that $(1-e^{-\delta_{n}t})h_{n}\leq(1-e^{-\delta
_{n}t})(h(t)+\varepsilon)$. Since $h(\infty)=0$,
there exists some $M>0$, such that for $t>M$, $h(t)\leq\varepsilon$ so
that $(1-e^{-\delta_{n}t})h_{n}\leq2\varepsilon$ for $t>M$.
For $t\leq M$, $(1-e^{-\delta_{n}t})h_{n}\leq(1-e^{-\delta_{n}M})(\Vert
h\Vert_{L^{\infty}}+\varepsilon)$ which is small if $\delta_{n}$ is small.
\end{pf}

%th12 #&#
\begin{theorem}\label{linearcase}
Assume $\lambda(z)=\nu+z$, $\nu>0$. $h(\cdot)$ satisfies the
assumptions in Lemma \ref{littleh} and $\int_{0}^{\infty}h(t)\,dt<1$. We have
%
%e5.6 #&#
\begin{equation}
\lim_{t\rightarrow\infty}\frac{1}{t}\log\mathbb{E}
\bigl[e^{\theta N_{t}}\bigr]=\nu(x-1),
\end{equation}
where $x$ is the minimal solution to $x=e^{\theta+\mu(x-1)}$, where $\mu
=\int_{0}^{\infty}h(t)\,dt$.
\end{theorem}

\begin{pf}
By Lemma \ref{ergodicformula}, we have
%
%e5.7 #&#
\begin{eqnarray}
&& \liminf_{t\rightarrow\infty}\frac{1}{t}\log\mathbb{E}
\bigl[e^{\theta N_{t}}\bigr]\nonumber
\\
&&\qquad \geq\sup_{(\hat{\lambda},\hat{\pi})\in
\mathcal{Q}_{e}}\int\bigl[\theta
\hat{\lambda}+\hat{\lambda} -\lambda-\hat{\lambda}\log(\hat{\lambda
}/\lambda)\bigr]\hat{\pi}\nonumber
\\
&&\qquad \geq \sup_{(K\lambda,\hat{\pi})\in\mathcal{Q}_{e},K\in\mathbb
{R}^{+}}\int\bigl[\theta\hat{\lambda}+\hat{\lambda}-
\lambda-\hat{\lambda}\log(\hat{\lambda}/\lambda) \bigr]\hat{\pi}
\nonumber\\[-8pt]\\[-8pt]
&&\qquad \geq\sup_{0<K<{1}/{\mu},(K\lambda,\hat{\pi})\in\mathcal
{Q}_{e}}\int\biggl[\theta+1-\frac{1}{K}-\log K
\biggr]\hat{\lambda}\hat{\pi}
\nonumber
\\
&&\qquad \geq \sup_{0<K<{1}/{\mu}} \biggl[\theta+1-\frac{1}{K}-\log K
\biggr]\cdot\frac{K\nu}{1-K\mu}
\nonumber
\\
&&\qquad = \cases{ \nu(x-1), &\quad if $\theta\in(-\infty,\mu-1-\log\mu]$,
\cr
+\infty, &\quad
otherwise,}
\nonumber
\end{eqnarray}
where $x$ is the minimal solution to $x=e^{\theta+\mu(x-1)}$.

By Lemma \ref{littleh}, we can find a sequence of $h_{n}$, where
$h_{n}(t)=\sum_{i=1}^{n}a_{i}e^{-b_{i}t}$ such that
$h_{n}\rightarrow h$ as $n\rightarrow\infty$ in both $L^{1}$ and
$L^{\infty}$ norms. Let $h_{\varepsilon}(t)=|h(t)-h_{n}(t)|$.
Then $0\leq h_{n}-h_{\varepsilon}\leq h\leq h_{n}+h_{\varepsilon}$.

Let $D_{1}$ be the set of points generated
by the Hawkes process with intensity $\lambda(\sum_{\tau\in D_{1},\tau
<t}h_{n}(t-\tau))$ and then conditional on $D_{1}$, let $D_{2}$ be the
set of points
generated by the point process with intensity
$\lambda(\sum_{\tau\in D_{1},\tau<t}(h_{n}+h_{\varepsilon})(t-\tau
))-\lambda(\sum_{\tau\in D_{1},\tau<t}h_{n}(t-\tau))$
and then iteratively, conditional on $D_{1},\ldots,D_{j-1}$, let
$D_{j}$ be the set of points generated by the point process with intensity
$\lambda(\sum_{\tau\in\bigcup_{i=1}^{j-1}D_{i},\tau<t}(h_{n}+h_{\varepsilon
})(t-\tau))
-\lambda(\sum_{\tau\in\bigcup_{i=1}^{j-2}D_{i},\tau<t}(h_{n}+h_{\varepsilon
})(t-\tau))$,
for any $j\geq3$. Let $D_{j}(t)$ correspond to the number of points in
$D_{j}$ by time $t$. Therefore,
$\sum_{j=1}^{\infty}D_{j}(t)$ equals the number of points generated by
Hawkes process with intensity $\lambda(\sum_{\tau<t}(h_{n}+h_{\varepsilon
})(t-\tau))$.
Our coupling argument is essentially the same as the one used in Br\'
{e}maud and Massouli\'{e} \cite{Bremaud}.
For a more formal treatment, one can use Poisson canonical space and
Poisson embeddings;
we refer to Br\'{e}maud and Massouli\'{e} \cite{Bremaud} for the details.

Assume that $\theta>0$, and we therefore have
%
%e5.8 #&#
\begin{equation}
\mathbb{E}\bigl[e^{\theta N_{t}}\bigr]\leq\mathbb{E} \bigl[e^{\theta\sum
_{j=1}^{\infty}D_{j}(t)}
\bigr].
\end{equation}
Now, for any $N\in\mathbb{N}$,
%
%e5.9 #&#
\begin{eqnarray}
&& \mathbb{E} \biggl[\exp\biggl\{\theta\sum_{j=1}^{N}D_{j}(t)\biggr\} \biggr]\nonumber
\\
&&\qquad = \mathbb{E} \Biggl[\exp\Biggl\{\theta\sum_{j=1}^{N-1}D_{j}(t)\Biggr\}\nonumber
\\
&&\hspace*{43pt}{}\times \exp\Biggl\{\bigl(e^{\theta
}-1\bigr)\int_{0}^{t}
\lambda\Biggl(\sum_{\tau\in\bigcup_{i=1}^{N-1}D_{i},\tau<s}(h_{n}+h_{\varepsilon
})(s-\tau) \Biggr)\nonumber
\\
&&\hspace*{105pt}{} -\lambda\Biggl(\sum_{\tau\in\bigcup_{i=1}^{N-2}D_{i},\tau
<s}(h_{n}+h_{\varepsilon})(s-\tau) \Biggr)\,ds\Biggr\} \Biggr]
\nonumber\\[-8pt]\\[-8pt]
&&\qquad \leq\mathbb{E} \Biggl[\exp\Biggl\{\theta\sum_{j=1}^{N-2}D_{j}(t)\Biggr\}\exp\bigl\{\bigl(\bigl(e^{\theta
}-1\bigr)\Vert h_{n}+h_{\varepsilon}\Vert_{L^{1}}+\theta\bigr)D_{N-1}(t)\bigr\} \Biggr]\nonumber
\\
&&\qquad \leq\cdots
\nonumber
\\
&&\qquad \leq\mathbb{E} \bigl[\exp\bigl\{\theta D_{1}(t)+f_{N-1}(\theta)D_{2}(t)\bigr\} \bigr]
\nonumber
\\
&&\qquad \leq\mathbb{E} \bigl[\exp\bigl\{\theta D_{1}(t)+\bigl(\exp\bigl\{f_{N-1}(\theta)\bigr\}-1\bigr)\Vert
h_{\varepsilon}\Vert_{L^{1}}D_{1}(t)\bigr\} \bigr],
\nonumber
\end{eqnarray}
where $f_{j}(\theta)=(e^{f_{j-1}(\theta)}-1)\Vert h_{n}+h_{\varepsilon
}\Vert_{L^{1}}+\theta$, for $j\geq2$ and $f_{1}(\theta)=\theta$.
Thus, for any $\theta\leq\Vert h_{n}+h_{\varepsilon}\Vert_{L^{1}}-1-\log
(\Vert h_{n}+h_{\varepsilon}\Vert_{L^{1}})$,
$e^{f_{N-1}(\theta)}$ converges to $y_{n}$ as $N\rightarrow\infty$,
where $y_{n}$ is the minimal solution to $y_{n}=e^{\theta+\Vert
h_{n}+h_{\varepsilon}\Vert_{L^{1}}(y_{n}-1)}$.
Since $D_{1}(t)$ is the Hawkes process with exciting function $h_{n}$,
%
%e5.10 #&#
\begin{equation}
\limsup_{t\rightarrow\infty}\frac{1}{t}\log\mathbb{E}
\bigl[e^{\theta N_{t}}\bigr]\leq\Gamma_{n}(p_{n}\theta),
\end{equation}
where $p_{n}=1+y_{n}\Vert h-h_{n}\Vert_{L^{1}}$. For $\Gamma
_{n}(p_{n}\theta)$, we have
\begin{eqnarray*}
\Gamma_{n}(p_{n}\theta)&=&\inf_{u\in\mathcal{U}_{p_{n}\theta}}\sup
_{(z_{1},\ldots,z_{n})\in\mathcal{Z}} \Biggl\{\frac{\mathcal
{A}u}{u}+\frac{p_{n}\theta}{\sum_{i=1}^{n}a_{i}}\sum
_{i=1}^{n}b_{i}z_{i}
\Biggr\}
\\
&\leq&\inf_{u=e^{\sum_{i=1}^{n}c_{i}z_{i}}\in\mathcal{U}_{p_{n}\theta
}}\sup_{(z_{1},\ldots,z_{n})\in\mathcal{Z}} \Biggl\{
\frac{\mathcal{A}u}{u}+\frac{p_{n}\theta}{\sum_{i=1}^{n}a_{i}}\sum
_{i=1}^{n}b_{i}z_{i}
\Biggr\}
\\
&=&\inf_{c_{1},\ldots,c_{n}}\sup_{(z_{1},\ldots,z_{n})\in\mathcal{Z}}
\Biggl\{-\sum
_{i=1}^{n}b_{i}c_{i}z_{i}+(
\nu+z_{1}+\cdots+z_{n}) \bigl(e^{\sum_{i=1}^{n}c_{i}a_{i}}-1 \bigr)
\\
&&\hspace*{212pt}{}+\frac{p_{n}\theta}{\sum_{i=1}^{n}a_{i}}\sum_{i=1}^{n}b_{i}z_{i}
\Biggr\}
\\
&=& \nu\bigl(e^{\sum_{i=1}^{n}c^{\ast}_{i}a_{i}}-1 \bigr)=\nu(x_{n}-1),
\end{eqnarray*}
where $c^{\ast}_{i}$ satisfies $-b_{i}c^{\ast}_{i}+e^{\sum
_{i=1}^{n}c^{\ast}_{i}a_{i}}-1+\frac{p_{n}\theta}{\sum_{i=1}^{n}a_{i}}b_{i}=0$,
for each $1\leq i\leq n$. By some computation, it is not hard to see
that $x_{n}=e^{\sum_{i=1}^{n}c^{\ast}_{i}a_{i}}$ satisfies
%
%e5.11 #&#
\begin{eqnarray}
x_{n}&=&\exp\Biggl\{p_{n}\theta+\sum
_{i=1}^{n}\frac{a_{i}}{b_{i}}(x_{n}-1)
\Biggr\}
\nonumber\\[-8pt]\\[-8pt]
&=&\exp\biggl\{\bigl(1+y_{n}\Vert h-h_{n}\Vert_{L^{1}}\bigr)
\theta+(x_{n}-1)\int_{0}^{\infty}h_{n}(t)\,dt
\biggr\}.
\nonumber
\end{eqnarray}
Since $h_{n}\rightarrow h$ in $L^{1}$ norm, it is not hard to see that
$x_{n}$ converges to the
minimal solution of $x=e^{\theta+\Vert h\Vert_{L^{1}}(x-1)}$ as
$n\rightarrow\infty$. If $\theta<0$, consider $h\geq h_{n}-h_{\varepsilon
}\geq0$ and
the argument is similar.
\end{pf}

%th13 #&#
\begin{theorem}
Assume $\lambda(z)=\nu+z$, $h\dvtx [0,\infty)\rightarrow\mathbb{R}^{+}$, $\mu:=\int_{0}^{\infty}h(t)\,dt<1$ and $h$ is continuous.
Then $(N_{t}/t\in\cdot)$ satisfies a large deviation principle with the
rate function $I(x)$ given by
%
%e5.12 #&#
\begin{equation}
I(x)= \cases{ \displaystyle x\log\biggl(\frac{x}{\nu+x\mu} \biggr)-x+\mu x+\nu,
&\quad if $x
\in[0,\infty)$,
\vspace*{5pt}\cr
+\infty, &\quad otherwise.}
\end{equation}
\end{theorem}

\begin{pf}
For the upper bound, apply the G\"{a}rtner--Ellis theorem. For the
lower bound, use the tilting method and identify $I(x)$ as
the Fenchel--Legendre transform of $\Gamma(\theta)$.
\end{pf}

%re14 #&#
\begin{remark}
In Bordenave and Torrisi \cite{Bordenave}, their $I(x)$ has the form
%
%e5.13 #&#
\begin{equation}
I(x)= \cases{\displaystyle x\theta_{x}+\nu-\frac{\nu x}{\nu+\mu x}, &\quad if $x\in[0,
\infty)$,
\vspace*{5pt}\cr
+\infty, &\quad otherwise,}
\end{equation}
where $\theta=\theta_{x}$ is the unique solution in $(-\infty,\mu-1-\log
\mu]$ of $\mathbb{E}[e^{\theta S}]=\frac{x}{\nu+x\mu}$, $x>0$.
Here, $\mathbb{E}[e^{\theta S}]$ satisfies the equation
%
%e5.14 #&#
\begin{equation}
\mathbb{E}\bigl[e^{\theta S}\bigr]=e^{\theta}\exp\bigl\{\mu\bigl(
\mathbb{E}\bigl[e^{\theta S}\bigr]-1\bigr) \bigr\},
\end{equation}
which implies that $\theta_{x}=\log(\frac{x}{\nu+x\mu} )-\mu(\frac
{x}{\nu+x\mu}-1 )$.
Substituting into the formula, their rate function is the same as what
we got.
\end{remark}

%re15 #&#
\begin{remark}
In Bordenave and Torrisi \cite{Bordenave}, the assumption in proving
the large deviations for linear Hawkes processes is slightly different
from ours. They did not require $h(\cdot)$ to be continuous, but they
further assumed that \mbox{$\int_{0}^{\infty}th(t)\,dt<\infty$.}
\end{remark}

%s6 #&#
\section{Large deviations for a special class of nonlinear Hawkes processes: An approximation approach}\label{sec6}

In this section, we prove the large deviation results for $(N_{t}/t\in
\cdot)$ for a very special class of
nonlinear $\lambda(\cdot)$ and $h(\cdot)$ that satisfies the
assumptions in Lemma \ref{littleh}.

Let $P_{n}$ denote the probability measure under which $N_{t}$ follows
the Hawkes process with
exciting function $h_{n}=\sum_{i=1}^{n}a_{i}e^{-b_{i}t}$ such that
$h_{n}\rightarrow h$ as $n\rightarrow\infty$
in both $L^{1}$ and $L^{\infty}$ norms. Let us define
%
%e6.1 #&#
\begin{equation}
\Gamma_{n}(\theta)=\lim_{t\rightarrow\infty}\frac{1}{t}\log
\mathbb{E}^{P_{n}} \bigl[e^{\theta N_{t}} \bigr].
\end{equation}

We have the following results.

%le16 #&#
\begin{lemma}\label{Lipschitz}
For any $K>0$ and $\theta_{1},\theta_{2}\in[-K,K]$, there exists some
constant $C(K)$ such that for any $n$,
%
%e6.2 #&#
\begin{equation}
\bigl|\Gamma_{n}(\theta_{1})-\Gamma_{n}(
\theta_{2})\bigr|\leq C(K)|\theta_{1}-\theta_{2}|.
\end{equation}
\end{lemma}

\begin{pf}
Without loss of generality, take $\theta_{2}>\theta_{1}$. Then
%
%e6.3 #&#
\begin{eqnarray}
\Gamma_{n}(\theta_{1})&\leq&\Gamma_{n}(
\theta_{2})\nonumber
\\
&=&\sup_{(\hat{\lambda},\hat{\pi})\in\mathcal{Q}^{\ast}_{e}}\int(\theta_{2}-
\theta_{1})\hat{\lambda}\hat{\pi} +\theta_{1}\hat{\lambda}
\hat{\pi}-\widehat{H}(\hat{\lambda},\hat{\pi})
\\
&\leq&\sup_{(\hat{\lambda},\hat{\pi})\in\mathcal{Q}^{\ast}_{e}}\int
(\theta_{2}-
\theta_{1})\hat{\lambda}\hat{\pi}+\Gamma_{n}(
\theta_{1}),
\nonumber
\end{eqnarray}
where
%
%e6.4 #&#
\begin{equation}
\mathcal{Q}_{e}^{\ast}= \biggl\{(\hat{\lambda},\hat{\pi})\in
\mathcal{Q}_{e}\dvtx \int\theta_{1}\hat{\lambda}\hat{\pi} -
\widehat{H}(\hat{\lambda},\hat{\pi})\geq\Gamma_{n}(\theta_{1})-1
\biggr\}.
\end{equation}
The key is to prove that $\sup_{(\hat{\lambda},\hat{\pi})\in\mathcal
{Q}_{e}^{\ast}}\int\hat{\lambda}\hat{\pi}\leq C(K)$
for some constant $C(K)>0$ depending only on $K$. Define
$u=u(z_{1},\ldots,z_{n})=e^{\sum_{i=1}^{n}c_{i}z_{i}}$ where
%
%e6.5 #&#
\begin{equation}
c_{i}=\frac{3K}{\sum_{i=1}^{n}({a_{i}}/{b_{i}})}\cdot\frac
{1}{b_{i}},\qquad1\leq i\leq
n.
\end{equation}
Define $V=-\frac{\mathcal{A}u}{u}$ such that
%
%e6.6 #&#
\begin{equation}
V(z_{1},\ldots,z_{n})=\frac{3K}{\sum_{i=1}^{n}({a_{i}}/{b_{i}})}\sum
_{i=1}^{n}z_{i}-\lambda(z_{1}+
\cdots+z_{n}) \bigl(e^{3K}-1\bigr).
\end{equation}
Notice that $\int\widehat{\mathcal{A}}f\hat{\pi}=0$ for any test function
$f$ with certain regularities.
If we try $f=\frac{z_{i}}{b_{i}}$, $1\leq i\leq n$, we get
%
%e6.7 #&#
\begin{equation}
-\int z_{i}\hat{\pi}+\frac{a_{i}}{b_{i}}\int\hat{\lambda}\hat{\pi}=0,
\qquad1\leq i\leq n.
\end{equation}
Summing over $1\leq i\leq n$, we get
%
%e6.8 #&#
\begin{equation}
\int\hat{\lambda}\hat{\pi}=\frac{1}{\sum_{i=1}^{n}
({a_{i}}/{b_{i}})}\int\sum
_{i=1}^{n}z_{i}\hat{\pi}.
\end{equation}
Notice that $\sum_{i=1}^{n}\frac{a_{i}}{b_{i}}=\Vert h_{n}\Vert
_{L^{1}}$ which is approximately $\Vert h\Vert_{L^{1}}$
when $n$ is large. Since $\limsup_{z\rightarrow\infty}\frac{\lambda
(z)}{z}=0$ and $\sum_{i=1}^{n}z_{i}\geq0$, we have
%
%e6.9 #&#
\begin{eqnarray}
\theta_{1}\int\hat{\lambda}\hat{\pi}&\leq& K\int\hat{\lambda}\hat{\pi}\nonumber
\\
&=& \frac{K}{\sum_{i=1}^{n}({a_{i}}/{b_{i}})} \int\sum_{i=1}^{n}z_{i}
\hat{\pi}
\\
&\leq& \frac{1}{2}\int V\hat{\pi}+C_{1/2}(K),\nonumber
\end{eqnarray}
where $C_{1/2}(K)$ is some positive constant depending only on $K$.

We claim that $\int V(z)\hat{\pi}\leq\widehat{H}(\hat{\pi})$ for any $\hat
{\pi}\in\mathcal{Q}_{e}^{\ast}$.
Let us prove it. By the ergodic theorem and Jensen's inequality,
%
%e6.10 #&#
\begin{eqnarray}
\int V(z)\hat{\pi} &=&\lim_{t\rightarrow\infty}\mathbb{E}^{\hat{\pi}}
\biggl[\frac{1}{t}\int_{0}^{t}V(Z_{s})\,ds
\biggr]
\nonumber\\[-8pt]\\[-8pt]
&\leq&\limsup_{t\rightarrow\infty}\frac{1}{t}\log
\mathbb{E}^{\pi} \bigl[e^{\int_{0}^{t}V(Z_{s})\,ds} \bigr]+\widehat{H}(\hat
{\pi}).\nonumber
\end{eqnarray}
Next, we will show that $u\geq1$. That is equivalent to proving $\sum
_{i=1}^{n}\frac{z_{i}}{b_{i}}\geq0$. Consider the process
%
%e6.11 #&#
\begin{equation}
Y_{t}=\sum_{i=1}^{n}
\frac{Z_{i}(t)}{b_{i}}=\sum_{\tau_{j}<t}\sum
_{i=1}^{n}\frac{a_{i}}{b_{i}}e^{-b_{i}(t-\tau_{j})}=\sum
_{\tau_{j}<t}g(t-\tau_{j}),
\end{equation}
where $g(t)=\sum_{i=1}^{n}\frac{a_{i}}{b_{i}}e^{-b_{i}t}$. Notice that
$g(t)=\int_{t}^{\infty}h(s)\,ds>0$.
Therefore, $Y_{t}\geq0$ almost surely and $\sum_{i=1}^{n}\frac
{Z_{i}(t)}{b_{i}}\geq0$.
Since\vspace*{1pt} $\frac{\mathcal{A}u}{u}+V=0$ and $u\geq1$, by the Feynman--Kac
formula and Dynkin's formula,
%
%e6.12 #&#
\begin{eqnarray}
\mathbb{E}^{\pi} \bigl[e^{\int_{0}^{t}V(Z_{s})\,ds} \bigr]&\leq&
\mathbb{E}^{\pi} \bigl[u(Z_{t})e^{\int_{0}^{t}V(Z_{s})\,ds} \bigr]\nonumber
\\
&=&u(Z_{0})+\int_{0}^{t}
\mathbb{E}^{\pi} \bigl[\bigl(\mathcal{A}u(Z_{s})+V(Z_{s})u(Z_{s})
\bigr)e^{\int_{0}^{s}V(Z_{u})\,du} \bigr]\,ds
\\
&=&u(Z_{0}),
\nonumber
\end{eqnarray}
and therefore $\int V(z)\hat{\pi}\leq\widehat{H}(\hat{\pi})$ for any $\hat
{\pi}\in\mathcal{Q}_{e}^{\ast}$. Hence
%
%e6.13 #&#
\begin{equation}
\theta_{1}\int\hat{\lambda}\hat{\pi}\leq\frac{1}{2}\int
V(z)+C_{1/2}(K)\leq\frac{1}{2}\widehat{H}+C_{1/2}(K).
\end{equation}
Notice that
%
%e6.14 #&#
\begin{equation}
-\infty<\Gamma_{n}(\theta_{1})-1\leq\theta_{1}
\int\hat{\lambda}\hat{\pi}-\widehat{H}\leq\Gamma_{n}(\theta_{1})<
\infty.
\end{equation}
Hence
%
%e6.15 #&#
\begin{equation}
\Gamma_{n}(\theta_{1})-1+\frac{1}{2}\widehat{H}\leq
\theta_{1}\int\hat{\lambda}\hat{\pi}-\frac{1}{2}\widehat{H}\leq
C_{1/2}(K),
\end{equation}
which implies $\widehat{H}\leq2(C_{1/2}(K)-\Gamma_{n}(\theta_{1})+1)$ and
so also,
%
%e6.16 #&#
\begin{eqnarray}
\int\hat{\lambda}\hat{\pi}&\leq&\frac{1}{2K}\int V\hat{\pi}+
\frac{1}{K}C_{1/2}(K)
\nonumber\\[-8pt]\\[-8pt]
&\leq&\frac{1}{K}\bigl(C_{1/2}(K)-
\Gamma_{n}(\theta_{1})+1\bigr)+\frac{1}{K}C_{1/2}(K).\nonumber
\end{eqnarray}
Finally, notice that since $h_{n}\rightarrow h$ in both $L^{1}$ and
$L^{\infty}$ norms, we can find a function $g$
such that $\sup_{n}h_{n}\leq g$ and $\Vert g\Vert_{L^{1}}<\infty$
and thus
%
%e6.17 #&#
\begin{equation}
\Gamma_{n}(\theta_{1})\geq\Gamma_{n}(-K)\geq
\Gamma_{g}(-K),
\end{equation}
where $\Gamma_{g}$ denotes the case when the rate function is still
$\lambda(\cdot)$ but the exciting function is $g(\cdot)$ instead of
$h_{n}(\cdot)$.
Notice that here $\Vert g\Vert_{L^{1}}<\infty$ but may not be less than
$1$. It is still well defined because of
the assumption $\lim_{z\rightarrow\infty}\frac{\lambda(z)}{z}=0$.
Indeed, we can find
$\lambda(z)=\nu_{\varepsilon}+\varepsilon z$ that dominates the original
$\lambda(\cdot)$ for $\nu_{\varepsilon}>0$ big enough
and $\varepsilon>0$ small enough so that $\varepsilon\Vert g\Vert_{L^{1}}<1$.
Now, we have $\Gamma_{g}(-K)\geq\Gamma^{\nu_{\varepsilon}}_{\varepsilon
g}(-K)$ which is finite (see Theorem \ref{linearcase}),
where $\Gamma^{\nu_{\varepsilon}}_{\varepsilon g}(-K)$ corresponds to the
case when $\lambda(z)=\nu_{\varepsilon}+\varepsilon z$. Hence
%
%e6.18 #&#
\begin{equation}
\sup_{(\hat{\lambda},\hat{\pi})\in\mathcal{Q}_{e}^{\ast}}\int\hat
{\lambda}\hat{\pi}\leq C(K),
\end{equation}
for some $C(K)>0$ depending only on $K$.
\end{pf}

%le17 #&#
\begin{lemma}\label{Cauchy}
Assume that $\lambda(\cdot)\geq c$ for some $c>0$, $\lim_{z\rightarrow
\infty}\frac{\lambda(z)}{z}=0$
and $\lambda(\cdot)^{\alpha}$ is Lipschitz with constant $L_{\alpha}$
for any $\alpha\geq1$.
Then for any $K>0$, $\Gamma_{n}(\theta)$ is Cauchy with $\theta$
uniformly in $[-K,K]$.
\end{lemma}

\begin{pf}
Let us write $H_{n}(t)=\sum_{\tau_{j}<t}h_{n}(t-\tau_{j})$. Observe
first, that for any $q$,
%
%e6.19 #&#
\begin{equation}
\qquad \exp\biggl\{q\int_{0}^{t}\log\biggl(
\frac{\lambda(H_{m}(s))}{\lambda(H_{n}(s))} \biggr)\,dN_{s} -\int_{0}^{t}
\biggl(\frac{\lambda(H_{m}(s))^{q}}{\lambda(H_{n}(s))^{q-1}}-\lambda
\bigl(H_{n}(s)\bigr) \biggr)\,ds \biggr
\}
\end{equation}
is a martingale under $P_{n}$. By H\"{o}lder's inequality, for any
$p,q>1$ with $\frac{1}{p}+\frac{1}{q}=1$,
%
%e6.20 #&#
\begin{eqnarray}
\qquad \mathbb{E}^{P_{m}}\bigl[e^{\theta N_{t}}\bigr]&=&\mathbb{E}^{P_{n}}
\biggl[e^{\theta N_{t}}\frac{dP_{m}}{dP_{n}} \biggr]\nonumber
\\
&=&\mathbb{E}^{P_{n}} \bigl[e^{\theta N_{t}-\int_{0}^{t}(\lambda
(H_{m}(s))-\lambda(H_{n}(s)))\,ds-\int_{0}^{t}
\log({\lambda(H_{n}(s))}/{\lambda(H_{m}(s))} )\,dN_{s}} \bigr]
\nonumber\\[-8pt]\\[-8pt]
&\leq&\mathbb{E}^{P_{n}} \bigl[e^{p\theta N_{t}-p\int_{0}^{t}(\lambda
(H_{m}(s))-\lambda(H_{n}(s)))\,ds} \bigr]^{1/p}\nonumber
\\
&&{}\times
\mathbb{E}^{P_{n}} \bigl[e^{q\int_{0}^{t}\log({\lambda
(H_{m}(s))}/{\lambda(H_{n}(s))} )\,dN_{s}} \bigr]^{1/q}.
\nonumber
\end{eqnarray}
By the Cauchy--Schwarz inequality,
%
%e6.21 #&#
\begin{eqnarray}
&& \mathbb{E}^{P_{n}} \bigl[e^{q\int_{0}^{t}\log({\lambda
(H_{m}(s))}/{\lambda(H_{n}(s))} )\,dN_{s}} \bigr]^{1/q}\nonumber
\\[-1pt]
&&\qquad \leq \mathbb{E}^{P_{n}} \bigl[e^{\int_{0}^{t} ({\lambda
(H_{m}(s))^{2q}}/{\lambda(H_{n}(s))^{2q-1}}-\lambda(H_{n}(s)) )\,ds} \bigr
]^{{1}/{(2q)}}
\nonumber\\[-9pt]\\[-9pt]
&&\qquad \leq \mathbb{E}^{P_{n}} \bigl[e^{({1}/{c^{2q-1}})L_{2q}\int
_{0}^{t}\sum_{\tau<s}|h_{m}(s-\tau)-h_{n}(s-\tau)|\,ds} \bigr]^{{1}/(2q)}
\nonumber
\\[-1pt]
&&\qquad \leq\mathbb{E}^{P_{n}} \bigl[e^{({1}/{c^{2q-1}})L_{2q}\Vert
h_{m}-h_{n}\Vert_{L^{1}}N_{t}} \bigr]^{{1}/(2q)}.
\nonumber
\end{eqnarray}
We also have
%
%e6.22 #&#
\begin{eqnarray}
&& \mathbb{E}^{P_{n}} \bigl[e^{p\theta N_{t}-p\int_{0}^{t}(\lambda
(H_{m}(s))-\lambda(H_{n}(s)))\,ds} \bigr]^{1/p}
\nonumber\\[-9pt]\\[-9pt]
&&\qquad \leq \mathbb{E}^{P_{n}} \bigl[e^{p\theta N_{t}+pL_{1}\Vert h_{m}-h_{n}\Vert
_{L^{1}}N_{t}} \bigr]^{1/p}.\nonumber
\end{eqnarray}
Therefore, by Lemma \ref{Lipschitz} and the fact $\Gamma_{n}(0)=0$ for
any $n$, we have
%
%e6.23 #&#
\begin{eqnarray}
&&\Gamma_{m}(\theta)-\Gamma_{n}(\theta)\nonumber
\\[-1pt]
&&\qquad \leq\frac{1}{p}\Gamma_{n} (p\theta+pL_{1}
\varepsilon_{m,n} ) +\frac{1}{2q}\Gamma_{n} \biggl(
\frac{L_{2q}\varepsilon_{m,n}}{c^{2q-1}} \biggr)-\Gamma_{n}(\theta)
\nonumber
\\[-1pt]
&&\qquad \leq C(K)L_{1}\varepsilon_{m,n}+\frac{C(K)}{2q}\cdot
\frac{L_{2q}\varepsilon_{m,n}}{c^{2q-1}}
\nonumber\\[-9pt]\\[-9pt]
&&\quad\qquad{}  +\frac{1}{p}\Gamma_{n}(p\theta)-
\frac{1}{p}\Gamma_{n}(\theta)+ \biggl(1-\frac{1}{p}
\biggr)\bigl|\Gamma_{n}(\theta)\bigr|,
\nonumber
\\[-1pt]
&&\qquad \leq C(K)L_{1}\varepsilon_{m,n}+\frac{C(K)}{2q}\cdot
\frac{L_{2q}\varepsilon_{m,n}}{c^{2q-1}}\nonumber
\\[-1pt]
&&\quad\qquad{} +\frac{C(K)(p-1)K}{p}+ \biggl
(1-\frac{1}{p} \biggr)C(K)K,
\nonumber
\end{eqnarray}
where $\varepsilon_{m,n}=\Vert h_{m}-h_{n}\Vert_{L^{1}}$. Hence,
%
%e6.24 #&#
\begin{equation}
\limsup_{m,n\rightarrow\infty}\bigl\{\Gamma_{m}(\theta)-
\Gamma_{n}(\theta)\bigr\}\leq2 \biggl(1-\frac{1}{p}
\biggr)C(K)K,
\end{equation}
which is true for any $p>1$. Letting $p\downarrow1$, we get the
desired result.
\end{pf}

%re18 #&#
\begin{remark}
If $\lambda(\cdot)\geq c>0$ and $\lim_{z\rightarrow\infty}\frac{\lambda
(z)}{z^{\alpha}}=0$ for any $\alpha>0$,
then, $\lambda(\cdot)^{\sigma}$ is Lipschitz for any $\sigma\geq1$.
For instance, $\lambda(z)=[\log(z+c)]^{\beta}$ satisfies the conditions
if $\beta>0$ and $c>1$.
\end{remark}

%th19 #&#
\begin{theorem}
Assume that $\lambda(\cdot)\geq c$ for some $c>0$, $\lim_{z\rightarrow
\infty}\frac{\lambda(z)}{z}=0$ and $\lambda(\cdot)^{\alpha}$ is
Lipschitz with constant $L_{\alpha}$ for any $\alpha\geq1$. Then, for
any $\theta\in\mathbb{R}$,
%
%e6.25 #&#
\begin{equation}
\lim_{t\rightarrow\infty}\frac{1}{t}\log\mathbb{E}
\bigl[e^{\theta N_{t}}\bigr]=\Gamma(\theta)=\lim_{n\rightarrow\infty}
\Gamma_{n}(\theta).
\end{equation}
\end{theorem}

\begin{pf}
By Lemma \ref{Cauchy}, $\Gamma_{n}(\theta)$ tends to $\Gamma(\theta)$
uniformly on any compact set $[-K,K]$. Since $\Gamma_{n}(\theta)$ is Lipschitz
by Lemma \ref{Lipschitz}, it is continuous and the limit $\Gamma$ is
also continuous. Let $\varepsilon_{n}=\Vert h_{n}-h\Vert_{L^{1}}\leq
\varepsilon$.
As in the proof of Lemma \ref{Cauchy}, for any $\theta\in[-K,K]$,
$p,q>1$, $\frac{1}{p}+\frac{1}{q}=1$, we get
%
%e6.26 #&#
\begin{eqnarray}
\qquad && \limsup_{t\rightarrow\infty}\frac{1}{t}\log\mathbb{E}
\bigl[e^{\theta N_{t}}\bigr]
\nonumber\\[-8pt]\\[-8pt]
&&\qquad \leq\Gamma_{n}(\theta)+C(K)L_{1}\varepsilon_{n}+
\frac{C(K)}{2q}\cdot\frac{L_{2q}\varepsilon_{n}}{c^{2q-1}}+2 \biggl(1-\frac{1}{p}
\biggr)C(K)K.
\nonumber
\end{eqnarray}
Letting $n\rightarrow\infty$ first and then $p\downarrow1$, we get
$\limsup_{t\rightarrow\infty}\frac{1}{t}\log\mathbb{E}[e^{\theta N_{t}}]
\leq\Gamma(\theta)$. Similarly, for any $p',q'>1$ with $\frac
{1}{p'}+\frac{1}{q'}=1$,
%
%e6.27 #&#
\begin{eqnarray}
\Gamma_{n}(\theta)&\leq&\liminf_{t\rightarrow\infty}
\frac{1}{pt}\log\mathbb{E}\bigl[e^{(p\theta+pL_{1}\varepsilon
_{n})N_{t}}\bigr]\nonumber
\\
&&{}  +\liminf
_{t\rightarrow\infty}\frac{1}{2qt}\log\mathbb{E} \bigl[e^{((
{L_{2q}\varepsilon_{n}})/{c^{2q-1}})N_{t}}
\bigr]
\nonumber\\[-8pt]\\[-8pt]
&\leq&\liminf_{t\rightarrow\infty}\frac{1}{pp't}\log\mathbb{E}
\bigl[e^{pp'\theta N_{t}}\bigr] +\liminf_{t\rightarrow\infty}\frac{1}{pq't}
\log\mathbb{E}\bigl[e^{q'pL_{1}\varepsilon_{n}N_{t}}\bigr]
\nonumber
\\
&&{}+\liminf_{t\rightarrow\infty}\frac{1}{2qt}\log\mathbb{E}
\bigl[e^{(({L_{2q}\varepsilon_{n}})/{c^{2q-1}})N_{t}} \bigr].
\nonumber
\end{eqnarray}
Since we can dominate $\lambda(\cdot)$ by the linear function $\lambda
(z)=\nu+z$ in which case the limit of logarithmic
moment generating function $\Gamma_{\nu}(\theta)$ is continuous in
$\theta$, we may let $n\rightarrow\infty$ to obtain
%
%e6.28 #&#
\begin{equation}
\Gamma(\theta)\leq\liminf_{t\rightarrow\infty}\frac{1}{pp't}\log
\mathbb{E}\bigl[e^{pp'\theta N_{t}}\bigr].
\end{equation}
This holds for any $\theta$ and thus
%
%e6.29 #&#
\begin{equation}
\liminf_{t\rightarrow\infty}\frac{1}{t}\log\mathbb{E}
\bigl[e^{\theta N_{t}}\bigr]\geq pp'\Gamma\biggl(\frac{\theta}{pp'}
\biggr).
\end{equation}
Letting $p,p'\downarrow1$ and using the continuity of $\Gamma(\cdot)$,
we get the desired result.
\end{pf}

%th20 #&#
\begin{theorem}
Assume that $\lambda(\cdot)\geq c$ for some $c>0$, $\lim_{z\rightarrow
\infty}\frac{\lambda(z)}{z}=0$
and $\lambda(\cdot)^{\alpha}$ is Lipschitz with constant $L_{\alpha}$
for any $\alpha\geq1$.
We have that $(N_{t}/t\in\cdot)$ satisfies the large deviation
principle with the rate function
%
%e6.30 #&#
\begin{equation}
I(x)=\sup_{\theta\in\mathbb{R}}\bigl\{\theta x-\Gamma(\theta)\bigr\}.
\end{equation}
\end{theorem}

\begin{pf}
For the upper bound, apply the G\"{a}rtner--Ellis theorem. Let us prove
the lower bound.
Let $B_{\varepsilon}(x)$ denote the open ball centered at $x$ with radius
$\varepsilon>0$. By H\"{o}lder's inequality,
for any $p,q>1$ with $\frac{1}{p}+\frac{1}{q}=1$,
%
%e6.31 #&#
\begin{equation}
P_{n} \biggl(\frac{N_{t}}{t}\in B_{\varepsilon}(x) \biggr) \leq
\bigg\Vert\frac{dP_{n}}{d\mathbb{P}} \bigg\Vert_{L^{p}(\mathbb{P})}\mathbb{P}
\biggl(\frac{N_{t}}{t}
\in B_{\varepsilon}(x) \biggr)^{1/q}.
\end{equation}
Therefore, letting $t\rightarrow\infty$, we have
%
%e6.32 #&#
\begin{eqnarray}
\sup_{\theta\in\mathbb{R}}\bigl\{\theta x-\Gamma_{n}(\theta)\bigr
\}&=&\lim_{t\rightarrow\infty}\frac{1}{t}\log P_{n} \biggl(
\frac{N_{t}}{t}\in B_{\varepsilon}(x) \biggr)\nonumber
\\
&\leq& \frac{1}{pp'}\Gamma\bigl(pp'L_{1}
\varepsilon_{n}\bigr)+\frac{1}{2pq'}\Gamma\biggl(\frac{L_{2pq'}\varepsilon
_{n}}{c^{2pq'-1}}
\biggr)
\\
&&{} +\frac{1}{q}\liminf_{t\rightarrow\infty}\frac{1}{t}\log
\mathbb{P} \biggl(\frac{N_{t}}{t}\in B_{\varepsilon}(x) \biggr),
\nonumber
\end{eqnarray}
where $\varepsilon_{n}=\Vert h_{n}-h\Vert_{L^{1}}$. Hence, letting
$n\rightarrow\infty$, see that
%
%e6.33 #&#
\begin{equation}
\frac{1}{q}\liminf_{t\rightarrow\infty}\frac{1}{t}\log
\mathbb{P} \biggl(\frac{N_{t}}{t}\in B_{\varepsilon}(x) \biggr)\geq\limsup
_{n\rightarrow\infty} \sup_{\theta\in\mathbb{R}}\bigl\{\theta x-
\Gamma_{n}(\theta)\bigr\}.
\end{equation}
Since $\Gamma_{n}(\theta)\rightarrow\Gamma(\theta)$ uniformly on any
compact set $K$,
%
%e6.34 #&#
\begin{equation}
\sup_{\theta\in K}\bigl\{\theta x-\Gamma_{n}(\theta)\bigr
\}\rightarrow\sup_{\theta\in K}\bigl\{\theta x-\Gamma(\theta)\bigr\},
\end{equation}
as $n\rightarrow\infty$ for any such set $K$. Notice that $\lambda(\cdot
)\geq c>0$ and recall that the limit for
the logarithmic moment generating function with parameter $\theta$ for
a Poisson process with constant rate $c$ is $(e^{\theta}-1)c$. Hence
%
%e6.35 #&#
\begin{equation}
\liminf_{\theta\rightarrow+\infty}\frac{\Gamma_{n}(\theta)}{\theta}\geq
\liminf
_{\theta\rightarrow+\infty}\frac{(e^{\theta}-1)c}{\theta}=+\infty,
\end{equation}
which implies that $\sup_{\theta\in\mathbb{R}}\{\theta x-\Gamma
_{n}(\theta)\}\rightarrow\sup_{\theta\in\mathbb{R}}\{\theta x-\Gamma
(\theta)\}$.
Therefore,
%
%e6.36 #&#
\begin{equation}
\frac{1}{q}\liminf_{t\rightarrow\infty}\frac{1}{t}\log
\mathbb{P} \biggl(\frac{N_{t}}{t}\in B_{\varepsilon}(x) \biggr) \geq\sup
_{\theta\in\mathbb{R}}\bigl\{\theta x-\Gamma(\theta)\bigr\}.
\end{equation}
Letting $q\downarrow1$, we get the desired result.
\end{pf}

%re21 #&#
\begin{remark}
The class of nonlinear Hawkes processes with general exciting function
$h$ for which
we proved the large deviation principle here is unfortunately a bit too special.
It works for the rate function like $\lambda(z)=[\log(c+z)]^{\beta}$,
for example,
but does not work for $\lambda(\cdot)$ that has sublinear power law growth.
In fact, by the coupling argument we used in the proof of the case of
linear $\lambda(\cdot)$ in Theorem \ref{linearcase},
we can prove that in the case when $\lim_{z\rightarrow\infty}\frac
{\lambda(z)}{z}=0$ and $\lambda(\cdot)$ is $\alpha$-Lipshcitz
and $\lambda(\cdot)\geq c>0$, $\Gamma(\theta)=\lim_{n\rightarrow\infty
}\Gamma_{n}(\theta)$ for $\theta\leq\mu-1-\log\mu$,
where $\mu=\int_{0}^{\infty}h(t)\,dt$ and $\Gamma$ and $\Gamma_{n}$ are
the limit of logarithmic moment generating functions
when the exciting functions are $h$ and $h_{n}$, respectively, and
$h_{n}\rightarrow h$ in $L^{1}$.
For the linear case, since $\Gamma(\theta)=\infty$ for $\theta>\mu
-1-\log\mu$, the coupling argument is good enough.
However, for the sublinear $\lambda(\cdot)$, $\Gamma(\theta)<\infty$
for any $\theta$ and the coupling argument is not enough.
In fact, it will appear in Zhu \cite{Zhu} that under the condition that
$\lim_{z\rightarrow\infty}\frac{\lambda(z)}{z}=0$, $\lambda(\cdot)$
is positive, increasing, $\alpha$-Lipshcitz and $\lambda(\cdot)\geq
c>0$ and $h(\cdot)$ is positive, decreasing and $\int_{0}^{\infty
}h(t)\,dt<\infty$,
there is a level-3 large deviation principle from which we can use the
contraction principle to get the level-1 large deviation principle
for $(N_{t}/t\in\cdot)$. Therefore, we conjecture that in the sublinear
case, $\Gamma(\theta)=\lim_{n\rightarrow\infty}\Gamma_{n}(\theta)$
for any $\theta$ and $(N_{t}/t\in\cdot)$ satisfies the large deviation
principle with
rate function $I(x)=\sup_{\theta\in\mathbb{R}}\{\theta x-\Gamma(\theta
)\}$.
The advantage of approximating the general case by the case when $h$ is
a sum of exponentials is that $\Gamma_{n}(\theta)$
can be evaluated by an optimization problem, which should be computable
by some numerical scheme.
\end{remark}

\section*{Acknowledgments}
The author is enormously grateful to his advisor Professor S. R. S.
Varadhan for suggesting this topic and for his superb guidance,
understanding, patience and generosity. He would also like to thank his
colleague Dmytro Karabash for valuable discussions on this project.
The author would also thank the anonymous referees who provided very
helpful suggestions for the improvement of this paper, and
to whom the author is much indebted. The author also thanks an
associate editor for helpful remarks. The author also thanks
Professor Henry McKean for pointing out some typos and minor mistakes
in the manuscript.


%suskaldyti doi

% imsref loaded by linak, 2014-03-31 13:34:10

\printaddresses

\end{document}